\documentclass[twoside,12pt]{article}
\usepackage{amsmath,amssymb,fdsymbol,eucal}
\usepackage{graphics}
\usepackage{epsfig}
\usepackage{pdfpages}
\usepackage{float}
\usepackage{times}
\usepackage{url}

\def\T{\mathcal{T}}
\def\hpsi{\hat{\psi}}
\begin{document}
\newtheorem{theorem}{Theorem}
\newtheorem{lemma}[theorem]{Lemma}
\newtheorem{corollary}[theorem]{Corollary}
\newtheorem{definition}[theorem]{Definition}
\newtheorem{example}[theorem]{Example}
\pagenumbering{roman}
\renewcommand{\thetheorem}{\thesection.\arabic{theorem}}
\renewcommand{\thelemma}{\thesection.\arabic{lemma}}
\newenvironment{proof}{\noindent{\bf{Proof.\/}}}{\hfill$\blacksquare$\vskip0.1in}
\renewcommand{\thetable}{\thesection.\arabic{table}}
\renewcommand{\thedefinition}{\thesection.\arabic{definition}}
\renewcommand{\theexample}{\thesection.\arabic{example}}
\renewcommand{\theequation}{\thesection.\arabic{equation}}
\newcommand{\mysection}[1]{\section{#1}\setcounter{equation}{0}
\setcounter{theorem}{0} \setcounter{lemma}{0}
\setcounter{definition}{0}}
\newcommand{\mrm}{\mathrm}
\newcommand{\be}{\begin{equation}}
\newcommand{\beq}{\begin{equation}}
\newcommand{\ee}{\end{equation}}
\newcommand{\eeq}{\end{equation}}
\newcommand{\ben}{\begin{enumerate}}
\newcommand{\een}{\end{enumerate}}

\title
{\bf PVTSI$^{\boldmath(m)}$: A Novel Approach to Computation of Hadamard Finite  Parts of  Nonperiodic Singular Integrals}

\author
{Avram Sidi\\
Computer Science Department\\
Technion - Israel Institute of Technology\\ Haifa 32000, Israel\\
E-mail:\quad  \url{asidi@cs.technion.ac.il}\\
URL:\quad    \url{http://www.cs.technion.ac.il/~asidi}}
\date{\today}
\bigskip\bigskip
\maketitle \thispagestyle{empty}
\newpage\noindent

\begin{abstract}
 We consider the numerical computation of
  $I[f]=\intBar^b_a f(x)\,dx$, the Hadamard Finite Part of the finite-range singular integral $\int^b_a f(x)\,dx$,    $f(x)=g(x)/(x-t)^{m}$ with $a<t<b$ and $m\in\{1,2,\ldots\},$  assuming that (i)\,$g\in C^\infty(a,b)$ and (ii)\,$g(x)$ is  allowed to have arbitrary  integrable singularities at the endpoints $x=a$ and $x=b$.
  We first prove that  $\intBar^b_a f(x)\,dx$ is invariant under any suitable variable transformation $x=\psi(\xi)$,  $\psi:[\alpha,\beta]\rightarrow[a,b]$, hence
  there holds  $\intBar^\beta_\alpha F(\xi)\,d\xi=\intBar^b_a f(x)\,dx$,  where $F(\xi)=f(\psi(\xi))\,\psi'(\xi)$. Based on this result, we next choose $\psi(\xi)$ such that the transformed integrand $F(\xi)$ is sufficiently periodic with period $\T=\beta-\alpha$,  and prove, with the help of some recent extension/generalization of  the Euler--Maclaurin expansion, that we can apply to
  $\intBar^\beta_\alpha F(\xi)\,d\xi$ the quadrature formulas derived for periodic singular integrals developed in an earlier work of the author. We give  a whole family of numerical quadrature formulas for $\intBar^\beta_\alpha F(\xi)\,d\xi$ for each $m$, which we  denote $\widehat{T}^{(s)}_{m,n}[{\cal F}]$,
  where ${\cal F}(\xi)$ is the $\T$-periodic extension of $F(\xi)$.
  Letting $G(\xi)=(\xi-\tau)^m F(\xi)$, with $\tau$ determined from $t=\psi(\tau)$,  and letting $h=\T/n$, for $m=3$, for example, we have the three formulas
  \begin{align*} \widehat{T}^{(0)}_{3,n}[{\cal F}]&=h\sum^{n-1}_{j=1}{\cal F}(\tau+jh)-\frac{\pi^2}{3}\,G'(\tau)\,h^{-1}
   +\frac{1}{6}\,G'''(\tau)\,h,\\
 \widehat{T}^{(1)}_{3,n}[{\cal F}]&=h\sum^n_{j=1}{\cal F}(\tau+jh-h/2)-\pi^2\,G'(\tau)\,h^{-1},\\
 \widehat{T}^{(2)}_{3,n}[{\cal F}]&=2h\sum^n_{j=1}{\cal F}(\tau+jh-h/2)-
\frac{h}{2}\sum^{2n}_{j=1}{\cal F}(\tau+jh/2-h/4).
\end{align*}
 We   show that all of the formulas $\widehat{T}^{(s)}_{m,n}[{\cal F}]$ converge to $I[f]$ as $n\to\infty$; indeed, if $\psi(\xi)$ is chosen such that ${\cal F}^{(i)}(\alpha)={\cal F}^{(i)}(\beta)=0$,
 $i=0,1,\ldots,q-1,$ and ${\cal F}^{(q)}(\xi)$ is absolutely integrable in every closed interval not containing $\xi=\tau$, then
  $$  \widehat{T}^{(s)}_{m,n}[{\cal F}]-I[f]=O(n^{-q})\quad\text{as $n\to\infty$},$$
  where $q$ is a positive integer determined by the behavior of $g(x)$ at $x=a$ and $x=b$ and also by $\psi(\xi)$. As such, $q$ can be increased arbitrarily (even to $q=\infty$) by choosing $\psi(\xi)$ suitably.
   We  provide several  numerical examples involving nonperiodic integrands and   confirm our theoretical results.
 \end{abstract}

\vspace{1cm} \noindent {\bf Mathematics Subject Classification 2010:}
 41A55,  41A60, 45B05, 45E05, 65B05, 65B15,  65D30, 65D32.

\vspace{1cm} \noindent {\bf Keywords and expressions:} Hadamard   Finite Part,
singular integrals, hypersingular integrals,
 supersingular integrals, Euler--Maclaurin expansions, asymptotic expansions,
 variable transformation,
 numerical quadrature, trapezoidal rule.

\thispagestyle{empty}
\newpage\
\pagenumbering{arabic}

\section{Introduction and background} \label{se1}
  \setcounter{equation}{0} \setcounter{theorem}{0}

Singular integrals $\int^b_a f(x)\,dx$ that do not exist in the regular sense but  are defined in the sense of  {\em Hadamard Finite Part (HFP)} arise in different areas of science and engineering, and  the numerical computation of their HFPs, denoted
\be \label{eq1} I[f]=\intBar^b_af(x)\,dx,\ee has been of considerable interest.
Of special interest are the integrals $\int^b_a f(x)\,dx$, whose integrands are of the
general form
\be \label{eq2} f(x) =\frac{g(x)}{(x-t)^m},\quad a<t<b,\quad m\in\{1,2,\ldots\}. \ee
The cases with $m=1,2,3$ occur in many applications and they are known as
{\em Cauchy Principal Value} integrals, {\em hypersingular} integrals, and
{\em supersingular} integrals, respectively.\footnote{We reserve the notation $\int^b_af(x)\,dx$ for integrals that exist in the regular sense. The  notation used for the Hadamard Finite Part of the integral $\int^b_a f(x)\,dx$ is $\intBar^b_a f(x)\,dx$ in general, while the accepted notation for the Cauchy Principal Value of the  integral $\int^b_a f(x)\,dx$  is $\intbar^b_a f(x)\,dx$.}
Many  numerical quadrature formulas for computing these three types of singular integrals
 can be found in the literature.

In the  papers Sidi and Israeli \cite{Sidi:1988:QMP} and  Sidi \cite{Sidi:2013:CNQ}, \cite{Sidi:2019:SSI-P1}, and  \cite{Sidi:2019:SSI-P2}, we derived and studied some interesting generalizations of the Euler--Maclaurin (E--M) expansion for singular integrals of the form described in \eqref{eq1}-\eqref{eq2}, and  we treated the special cases of $m=1,2,3$ in detail. Based on these generalized E--M expansions, we developed numerical quadrature formulas for the case in which the $T$-periodic extension of $f(x)$---which, we denote also by $f(x)$---is infinitely differentiable for all $x\in\mathbb{R}_t$, that is,  $f\in C^\infty(\mathbb{R}_t)$, with
\beq\label{eqRt} T=b-a,\quad \mathbb{R}_t=\mathbb{R} \setminus\{t+kT\}^\infty_{k=-\infty}.\eeq
  All these quadrature formulas are very effective and enjoy {\em spectral} accuracy. In view of this, one may ask as to whether they will continue to be effective  when  the  $T$-periodic extension of $f(x)$ fails to be infinitely differentiable on $\mathbb{R}_t$. This is precisely the issue we  address in this work by  relaxing considerably the condition that
   $f\in C^\infty(\mathbb{R}_t)$.

We assume throughout this work that  $f(x)$ in \eqref{eq1} is as in
\begin{gather} f(x) =\frac{g(x)}{(x-t)^m},\quad a<t<b,\quad m\in\{1,2,\ldots\}, \notag\\
g\in C^\infty(a,b),\quad g(x)\ \text{integrable at $x=a$ and $x=b$.} \label{eq3}\end{gather}
We note first that $g(x)$ being integrable at $x=a$ and $x=b$ is the same as $f(x)$ being integrable at $x=a$ and $x=b$.
Next, we note  that we are {\em not}  imposing on $g(x)$ [equivalently, on $f(x)$]  differentiability or even continuity conditions
at $x=a$ and $x=b$. Summarizing, the functions $f(x)$ treated in this work satisfy the following conditions:
\begin{itemize}
\item [(i)]
they are in $C^\infty((a,t)\cup(t,b))$,
\item [(ii)]
they have  a {\em nonintegrable polar singularity} at $x=t$, and
\item [(iii)]
they are allowed to have  {\em arbitrary integrable  singularities}  at $x=a$ and $x=b$.
\end{itemize}

Our approach to the numerical treatment of the integrals $\intBar^b_a f(x)\,dx$ under \eqref{eq3} proceeds in two steps: (i)\,First, we periodize the integrands $f(x)$ in \eqref{eq3} in some sense by using  suitable variable  transformations. (ii)\,Next, we use the appropriate  quadrature formulas    developed in \cite{Sidi:1988:QMP},
 \cite{Sidi:2013:CNQ}, and  \cite{Sidi:2019:SSI-P1} on the transformed integrals.

There are, however, three major questions related to this approach that need to be addressed:
\begin{enumerate}
\item We know that, if $\int^b_a u(x)\,dx$ exists as a regular integral, a legitimate variable transformation will not change its value. Can we guarantee that this will be the case also for the  HFP integrals $\intBar^b_af(x)\,dx$ considered here, which do not exist in the regular sense? This question is relevant since HFP integrals have most, but not all, of the properties of regular integrals and some properties that are quite
unusual. For example, they are invariant with respect to translation of the variable of integration $x$, but they are not necessarily invariant
under a nonlinear or even linear scaling of $x$. To see  this, let us consider the HFP integral $\intBar^1_0dx/x=\!0$
 given in Davis and Rabinowitz \cite[p. 13]{Davis:1984:MNI}.\\
  $\bullet$ Following the {\em nonlinear} scaling variable transformation $x=y^2/2$, the resulting HFP integral is $2\intBar^{\sqrt{2}}_0dy/y=\log 2.$ \\
 $\bullet$ Following the {\em linear} scaling variable transformation $x=2y$, the resulting HFP integral is     $\intBar^{1/2}_0dy/y=-\log 2.$

  The variable transformations we will be using are {\em nonlinear} scalings of $x$.
\item
Does the variable transformation change the nature of the singularity at $x=t$?
If so, in what way?
\item
The quadrature formulas of  \cite{Sidi:1988:QMP}, \cite{Sidi:2013:CNQ}, and  \cite{Sidi:2019:SSI-P1} have spectral accuracy when $g\in C^\infty[a,b]$ and $f(x)$ is $T$-periodic and $f\in C^\infty(\mathbb{R}_t)$. Can we guarantee that they will be effective when either (i)\,the periodic extension of $f(x)$ is not infinitely differentiable on $\mathbb{R}_t$, or (ii)\,$g(x)$ is not infinitely differentiable, or differentiable at all,  at $x=a$ and/or $x=b$?
\end{enumerate}
The answer to the first question is yes if $f(x)$ is as in \eqref{eq3}. We give a detailed proof of this  in Theorem \ref{thpsi} in Section~\ref{se3}. The answer to the second question is no, as we show again in Section \ref{se3}; the singularity in the transformed integrand remains a pole of order $m$ because $a<t<b$.
The answer to the third question
is  yes provided we use
suitable variable transformations, and this is the subject of Theorem \ref{thw2} in Section~\ref{se5}.

In Section \ref{se2}, we give a brief description of the quadrature methods
developed in \cite{Sidi:2019:SSI-P1}
for the singular integrals in \eqref{eq1}, in case $f(x)$ is $T$-periodic and
infinitely differentiable for all $x$, except at $x=t+kT,$ $k=0,\pm1,\pm2,\ldots.$
In Section \ref{se3}, we provide a detailed analysis of the singular integrals in \eqref{eq1} and \eqref{eq3}
under  legitimate variable transformations.
In Section~\ref{se4}, we  discuss the issue of periodization of the integrand $f(x)$ via  suitable variable transformations and explore the analytical behavior of the transformed integrand in detail.
In Section \ref{se5}, we develop the quadrature formulas of this work for the nonperiodic singular integrals in \eqref{eq1}, where the integrands $f(x)$ are as in \eqref{eq3}. These formulas are based on a refined asymptotic analysis of the transformed integrand, followed by the application of Theorem  \ref{thA} in the appendix to this work that extends a generalized Euler--Maclaurin expansion due to the author given in \cite{Sidi:2012:EME-P1}. We note that this appendix forms an integral part of this work. We will call the approach leading to the quadrature formulas thus developed  {\em Perodizing Variable Transformed Singular Integration} and will denote it  {\em PVTSI\,$^{(m)}$} for short.
Finally, in Section~\ref{se6}, we provide numerical examples that illustrate the use of the approach proposed and confirm the theoretical results of this paper.

Before proceeding further, we note the following facts concerning  the Riemann Zeta function $\zeta(z)$, which we will need later:
$$\zeta(-2k)=0,\ \ k=1,2,\ldots;\quad
\zeta(2k)=(-1)^{k+1}\frac{(2\pi)^{2k}}{2(2k)!}B_{2k}>0,\quad k=0,1,\ldots.$$
Here $B_s$ are the  Bernoulli numbers.
Hence $\zeta(2k)$,  $k=0,1,\ldots,$ are all known; for example,
$$ \zeta(0)=-\frac{1}{2},\quad \zeta(2)=\frac{\pi^2}{6},\quad \zeta(4)=\frac{\pi^4}{90}, \quad \text{and so on.}$$

\section{Review of numerical quadrature formulas for periodic \\ singular  integrals} \label{se2}
\setcounter{equation}{0} \setcounter{theorem}{0}

\subsection{Review of numerical quadrature formulas for arbitrary  $m$} \label{sse22}
 In  \cite{Sidi:2019:SSI-P1}, we developed the following numerical quadrature formulas
 for the HFP integrals $\intBar^b_a f(x)\,dx$, where $f(x)$ are as in \eqref{eq2} with arbitrary integer $m$,  are $T$-periodic and belong to  $C^\infty(\mathbb{R}_t)$, with $T$ and $\mathbb{R}_t$ as in \eqref{eqRt}:
\begin{itemize}
\item
For even $m$, $m=2r$, $r=1,2,\ldots,$ and with $h=T/n$, we have
\beq \label{eqT0meven}\widehat{T}^{(0)}_{2r,n}[f]=h\sum^{n-1}_{j=1}f(t+jh)
-2\sum^r_{i=0}\frac{g^{(2i)}(t)}{(2i)!}\,\zeta(2r-2i)\,h^{-2r+2i+1}.\eeq
\item
For odd $m$,  $m=2r+1,$  $r=0,1,\ldots,$  and with $h=T/n$, we have
\beq \label{eqT0modd} \widehat{T}^{(0)}_{2r+1,n}[f]=h\sum^{n-1}_{j=1}f(t+jh)
-2\sum^r_{i=0}\frac{g^{(2i+1)}(t)}{(2i+1)!}\,\zeta(2r-2i)\,h^{-2r+2i+1}.\eeq
\end{itemize}

We also proved that, as $n\to\infty$,  $\widehat{T}^{(0)}_{m,n}[f]\to I[f]$   spectrally, that is,
\beq\label{eqT0conv} \widehat{T}^{(0)}_{m,n}[f]-I[f]=o(n^{-\mu})\quad\text{as $n\to\infty$}\quad \forall \,\mu>0.\eeq

\sloppypar

In addition, we  showed that,
with the $\widehat{T}^{(0)}_{m,n}[f]$ available, we can construct the numerical quadrature formulas
$\widehat{T}^{(s)}_{m,n}[f]$, $s=1,2,\ldots,\lfloor \frac{m+2}{2}\rfloor,$ by performing $s$ steps of  a ``Richardson-like extrapolation'' process on the relevant sequences $\{\widehat{T}^{(0)}_{m,2^kn}[f]\}^{s}_{k=0}$, by which we eliminate the powers $h^1, h^{-1},h^{-3},\ldots,$ in this order, from $\widehat{T}^{(0)}_{m,n}[f]$.
This also amounts to eliminating the  $g^{(p)}(t)$ from $\widehat{T}^{(0)}_{m,n}[f]$ one by one,
starting from the highest order derivative and down. Thus, we eliminate
$g^{(m)}(t),g^{(m-2)}(t),\ldots,g^{(2)}(t),g^{(0)}(t),$ for even $m$
and $g^{(m)}(t),g^{(m-2)}(t),\ldots,g^{(3)}(t),g^{(1)}(t)$ for odd $m$.
For example, with $s=1,2,3$,  we have
$$\widehat{T}^{(1)}_{m,n}[f]=2\widehat{T}^{(0)}_{m,2n}[f]-\widehat{T}^{(0)}_{m,n}[f],$$
\begin{align*}
\widehat{T}^{(2)}_{m,n}[f]&=2\widehat{T}^{(1)}_{m,n}[f]-\widehat{T}^{(1)}_{m,2n}[f]\\
&=-2\widehat{T}^{(0)}_{m,n}[f]+5\widehat{T}^{(0)}_{m,2n}[f]-2\widehat{T}^{(0)}_{m,4n}[f],
\end{align*}
\begin{align*} \widehat{T}^{(3)}_{m,n}[f]&=\frac{8}{7}\widehat{T}^{(2)}_{m,n}[f]-\frac{1}{7}\widehat{T}^{(2)}_{m,2n}[f]\\
&= \frac{16}{7}\widehat{T}^{(1)}_{m,n}[f]-\frac{10}{7}\widehat{T}^{(1)}_{m,2n}[f]+
\frac{2}{7}\widehat{T}^{(1)}_{m,4n}[f]\\
&=-\frac{16}{7}\widehat{T}^{(0)}_{m,n}[f]+6\widehat{T}^{(0)}_{m,2n}[f]
-3\widehat{T}^{(0)}_{m,4n}[f]+\frac{2}{7}\widehat{T}^{(0)}_{m,8n}[f].\end{align*}
In general, eliminating only the powers $h^1,h^{-1},h^{-3},\ldots, h^{-2s+3},$ we have
\be\label{eqalpha} \widehat{T}^{(s)}_{m,n}[f]=\sum^s_{k=0}\alpha^{(s)}_{m,k} \widehat{T}^{(0)}_{m,2^kn}[f],\quad
\sum^s_{k=0}\alpha^{(s)}_{m,k}=1;\quad \text{$\alpha^{(s)}_{m,k}$ independent of $n$}.\ee

 Concerning the formulas $\widehat{T}^{(s)}_{m,n}[f]$, we have the following general convergence theorem:

\begin{theorem}\label{thw1t}
If $f(x)$ is as in \eqref{eq2}, $T$-periodic, and infinitely differentiable for all $x\in\mathbb{R}_t$, then all the numerical quadrature formulas $\widehat{T}^{(s)}_{m,n}[f]$ in \eqref{eqalpha} converge to $I[f]=\intBar^b_af(x)\,dx$ with spectral accuracy, namely,
\be \widehat{T}^{(s)}_{m,n}-I[f]=o(n^{-\mu})\quad\text{as $n\to\infty$} \quad\forall\,\mu>0.\ee
In words, the errors $\widehat{T}^{(s)}_{m,n}[f]-I[f]$ tend to zero as $n\to\infty$ faster than every negative power of $n$.
\end{theorem}

\subsection{Review of the cases $m=1,2,3,4$}
For $m=1,2,3,$ the formulas above assume the following specific forms:

\begin{enumerate}
\item
The case $m=1$:
\begin{subequations}
\begin{align} \widehat{T}^{(0)}_{1,n}[f]&=h\sum^{n-1}_{j=1}f(t+jh)+g'(t)h, \label{eqT10}\\
 \widehat{T}^{(1)}_{1,n}[f]&=h\sum^{n}_{j=1}f(t+jh-h/2).\label{eqT11}\end{align}
 \end{subequations}
\item
The case $m=2$:
\begin{subequations}
\begin{align} \widehat{T}^{(0)}_{2,n}[f]&=h\sum^{n-1}_{j=1}f(t+jh)-\frac{\pi^2}{3}g(t)h^{-1}
+\frac{1}{2}g''(t)h, \label{eqT20}\\
 \widehat{T}^{(1)}_{2,n}[f]&=h\sum^{n}_{j=1}f(t+jh-h/2) -\pi^2g(t)h^{-1},\label{eqT21} \\
\widehat{T}^{(2)}_{2,n}[f]&=2h\sum^{n}_{j=1}f(t+jh-h/2) -
\frac{h}{2}\sum^{2n}_{j=1}f(t+jh/2-h/4).\label{eqT22}
 \end{align}
 \end{subequations}
 \item
The case $m=3$:
\begin{subequations}
\begin{align} \widehat{T}^{(0)}_{3,n}[f]&=h\sum^{n-1}_{j=1}f(t+jh)-\frac{\pi^2}{3}g'(t)h^{-1}
+\frac{1}{6}g'''(t)h, \label{eqT30}\\
 \widehat{T}^{(1)}_{3,n}[f]&=h\sum^{n}_{j=1}f(t+jh-h/2) -\pi^2g'(t)h^{-1}, \label{eqT31}\\
 \widehat{T}^{(2)}_{3,n}[f]&=2h\sum^{n}_{j=1}f(t+jh-h/2) -
\frac{h}{2}\sum^{2n}_{j=1}f(t+jh/2-h/4).\label{eqT32}\end{align}
\end{subequations}
\item
{\em The case $m=4$}:
\begin{subequations}
\begin{align}
\widehat{T}^{(0)}_{4,n}[f]&=h\sum^{n-1}_{j=1}f(t+jh)-\frac{\pi^4}{45}g(t)h^{-3}
-\frac{\pi^2}{6}g''(t)h^{-1}+\frac{1}{24}g^{(4)}(t)h, \label{eqT40} \\
\widehat{T}^{(1)}_{4,n}[f]&=h\sum^{n}_{j=1}f(t+jh-h/2)-\frac{\pi^4}{3}g(t)h^{-3}
-\frac{\pi^2}{2}g''(t)h^{-1}, \label{eqT41} \\
\widehat{T}^{(2)}_{4,n}[f]&=2h\sum^{n}_{j=1}f(t+jh-h/2)
-\frac{h}{2}\sum^{2n}_{j=1}f(t+jh/2-h/4)+2\pi^4g(t)h^{-3} \label{eqT42},  \\
\widehat{T}^{(3)}_{4,n}[f]&=\frac{16h}{7}\sum^{n}_{j=1}f(t+jh-h/2)
-\frac{5h}{7}\sum^{2n}_{j=1}f(t+jh/2-h/4)\notag \\
&\hspace{4.5cm}+\frac{h}{28}\sum^{4n}_{j=1}f(t+jh/4-h/8). \label{eqT43}
\end{align}
\end{subequations}
\end{enumerate}

These formulas are derived  and studied in
\cite{Sidi:1988:QMP} (for $m=1$), in \cite{Sidi:2013:CNQ} (for $m=1,2$), and  in \cite{Sidi:2019:SSI-P1}, \cite{Sidi:2019:SSI-P2} (for $m=3$).
Concerning the formulas, we have the following convergence theorem
that strengthens Theorem \ref{thw1t}:

\begin{theorem}\label{thw2t}
If $f(z)$ is both $T$-periodic and analytic in a strip $D_\sigma$ of the complex $z$-plane,
$$D_\sigma=\{z\in \mathbb{C}:\ |\Im z|<\sigma\},$$ then,  for $m=1,2,3$, we have
\be \widehat{T}^{(s)}_{m,n}-I[f]=O(e^{-2n\pi \rho/T})\quad\text{as $n\to\infty$},\quad
\forall \rho<\sigma.\ee
\end{theorem}

Thus, practically speaking, we have
$$ \widehat{T}^{(s)}_{m,n}-I[f]=O(e^{-2n\pi \sigma/T})\quad\text{as $n\to\infty$}.$$
 For the proof of this result, see \cite{Sidi:1988:QMP} for $m=1$, \cite{Sidi:2013:CNQ} for $m=2$, and \cite{Sidi:2019:SSI-P2} for $m=3$.

\section{Variable transformations and singular integrals} \label{se3}
\setcounter{equation}{0} \setcounter{theorem}{0}
  Theorem \ref{thpsi} that follows  shows that the HFP integrals in \eqref{eq3} are invariant under a variable transformation $x=\psi(\xi)$ provided $g(x)$ and $\psi(\xi)$  have enough continuous derivatives.

\begin{theorem}\label{thpsi}
Let $m$ be a positive integer, and let
\be\label{eqpsi1}f(x) =\frac{g(x)}{(x-t)^m}, \quad a<t<b,\quad
g\in C^m(a,b),\quad g(x)\ \text{integrable at $x=a$ and $x=b$}, \ee
and let the variable transformation $x=\psi(\xi)$ be such that
\begin{gather}\label{eqpsi2}\psi: [\alpha,\beta]\rightarrow [a,b];\quad  \psi(\alpha)=a,\quad \psi(\beta)=b,\notag\\
\psi\in C^m[\alpha,\beta]; \quad \psi'(\xi)>0\quad\text{for \ $\alpha<\xi<\beta$}. \label{eqpsi22}\end{gather}
Then
\be\label{eqpsi3}
\intBar_\alpha^\beta {f}(\psi(\xi))\,\psi'(\xi)\,d\xi=\intBar^b_a {f}(x)\,dx
\quad\text{independent of $\psi(\xi)$.}\ee
\end{theorem}

\noindent{\bf Remarks.}
\begin{enumerate}
\item
The differentiability conditions imposed on $g(x)$ and $\psi(\xi)$ seem to be minimal possible. Of course, the theorem is correct also when $g\in C^p(a,b)$ and $\psi\in C^q[\alpha,\beta]$, for all  $p,q\geq m$.
The result for $m=1$ (Cauchy Principal Value) is not new; see Gakhov \cite[p. 17]{Gakhov:1966:BVP}; we provide a proof of this case for completeness.

\item
Let us denote the transformed integrand by $F(\xi)$, that is,
\beq F(\xi)={f}(\psi(\xi))\,\psi'(\xi).\eeq
It is easy to see that
$F(\xi)$ has the same kind of singularity structure as $f(x)$; actually, $F(\xi)$  is of the form
\begin{gather} F(\xi)=\frac{G(\xi)}{(\xi-\tau)^m},\quad G(\xi)=\frac{g(\psi(\xi))}
{(\psi[\xi,\tau])^m}\,\psi'(\xi),\notag\\ \psi[\xi,\tau]=\frac{\psi(\xi)-\psi(\tau)}{\xi-\tau} \neq0\quad \text{for $\xi\neq\tau$},\quad \psi[\tau,\tau]=\psi'(\tau)>0,\label{eqpsi43}\end{gather}
$\tau\in (\alpha,\beta)$ being the unique solution of the equation $t=\psi(\xi)$ for $\xi$ since $\psi'(\xi)>0$ on $(\alpha, \beta)$.\footnote{Given $t\in(a,b)$, we can determine $\tau$ as the solution to the equation $\theta(\xi)=0$ with $\theta(\xi)=\psi(\xi)-t$, which can be achieved by  using  the Newton--Raphson  method, for example. For some of the variable transformations we present later in subsection \ref{sse43}, given $t$, $\tau$ is readily available, however.}
Consequently,  we also have
\be  \label{eqpsi44} G(\tau)=\frac{g(\psi(\tau))}{[\psi'(\tau)]^{m-1}}=\frac{g(t)}{[\psi'(\tau)]^{m-1}}. \ee
$G^{(i)}(\tau)$, $i\geq1$,  can be obtained  by differentiating $G(\xi)$ in \eqref{eqpsi43}  and letting $\xi\to\tau$. Thus,
\be
\label{eqpsi45} G'(\tau)=\frac{g'(t)}{[\psi'(\tau)]^{m-2}}+\bigg(1-\frac{m}{2}\bigg)
\frac{g(t)\psi''(\tau)}{[\psi'(\tau)]^{m}},\ee
for example.

\item
We recall that if $u(x)$ has a nonintegrable singularity at $x=t$ for $t\in(a,b)$ but is integrable on any subinterval of $(a,b)$ that does not contain $x=t$, then $\intBar^b_a u(x)\,dx$ is obtained
 by expanding
$$\phi(\epsilon)=\int^{t-\epsilon}_au(x)\,dx+\int^b_{t+\epsilon}u(x)\,dx, \quad \epsilon>0,$$
asymptotically as  $\epsilon\to0$,  discarding those terms that go to infinity, and retaining the limit  of the remaining terms, as  $\epsilon\to0$. (See Monegato \cite{Monegato:2009:DPA}, for example.)
\end{enumerate}

\begin{proof}
Let us express ${f}(x)$ in the form
$$
 {f}(x)= w(x)+\sum^{m-1}_{i=0}\frac{g^{(i)}(t)}{i!}\frac{1}{(x-t)^{m-i}},$$
where
$$w(x)=\frac{g(x)-\displaystyle\sum^{m-1}_{i=0}\frac{g^{(i)}(t)}{i!}(x-t)^i}{(x-t)^{m}}\quad\text{when $x\neq t$},\quad w(t)=\frac{g^{(m)}(t)}{m!}.$$
Clearly, $w(x)$ is continuous on $(a,b)$ and integrable at $x=a$ and $x=b$.

Now, for each $t\in(a,b)$, there is a unique $\tau\in(\alpha,\beta)$ such that $t=\psi(\tau)$, as already explained above.
Therefore,
\be\label{eqs2-1}
\intBar_\alpha^\beta {f}(\psi(\xi))\,\psi'(\xi)\,d\xi=
\intBar_\alpha^\beta w(\psi(\xi))\,\psi'(\xi)\,d\xi +\sum^{m-1}_{i=0}\frac{g^{(i)}(t)}{i!}\intBar_\alpha^\beta
\frac{\psi'(\xi)}{(\psi(\xi)-\psi(\tau))^{m-i}}\,d\xi.\ee

First, because $w(x)$ is  continuous on $(a,b)$ and integrable at $x=a$ and $x=b$,  we have that
$w(\psi(\xi))\psi'(\xi)$ is continuous on $(\alpha,\beta)$ and integrable at $\xi=\alpha$ and $\xi=\beta$. Consequently,
\be \label{eqg1}
\intBar_\alpha^\beta w(\psi(\xi))\,\psi'(\xi)\,d\xi =\int_\alpha^\beta w(\psi(\xi))\,\psi'(\xi)\,d\xi =
\int^b_a w(x)\,dx. \ee

Next, for  $k=1,2,\ldots,m,$ let us consider
\be\phi_{k}(\epsilon)=\int^{\tau-\epsilon}_\alpha\frac{\psi'(\xi)}{(\psi(\xi)-\psi(\tau))^{k}}\,d\xi+
\int^\beta_{\tau+\epsilon}\frac{\psi'(\xi)}{(\psi(\xi)-\psi(\tau))^{k}}\,d\xi.\ee
In what follows, we make repeated use of the facts that $\psi(\alpha)=a$, $\psi(\beta)=b$, and $\psi(\tau)=t$.

For $k=1$, we have
$$ \phi_1(\epsilon)=\log\bigg|\frac{\psi(\tau-\epsilon)-\psi(\tau)}{\psi(\alpha)-\psi(\tau)}\bigg|
+\log\bigg|\frac{\psi(\beta)-\psi(\tau)}{\psi(\tau+\epsilon)-\psi(\tau)}\bigg|,$$
hence
$$\phi_1(\epsilon)=\log\bigg|\frac{b-t}{a-t}\bigg| +\Lambda_1(\epsilon),$$
where
$$\Lambda_1(\epsilon)=\log\bigg|\frac{\psi(\tau-\epsilon)-\psi(\tau)}
{\psi(\tau+\epsilon)-\psi(\tau)}\bigg|.
$$
Application of L'H\^{o}spital's rule results in $\lim_{\epsilon\to0+}\Lambda_1(\epsilon)=0$. Thus,
\be\label{eqg2}\intBar_\alpha^\beta  \frac{\psi'(\xi)}{\psi(\xi)-\psi(\tau)}\,d\xi=\log\bigg|\frac{b-t}{a-t}\bigg|,\ee
independent of $\psi(\xi)$.

For  $k=2,3,\ldots,m,$  we have
\begin{align*}\phi_{k}(\epsilon)= -\frac{1}{k-1}\bigg\{ &\bigg[\frac{1}{(\psi(\tau-\epsilon)-\psi(\tau))^{k-1}}-\frac{1}{(\psi(\alpha)-\psi(\tau))^{k-1}}\bigg]\\
+&\bigg[\frac{1}{(\psi(\beta)-\psi(\tau))^{k-1}}-\frac{1}{(\psi(\tau+\epsilon)-\psi(\tau))^{k-1}}\bigg]\bigg\},
\end{align*}
hence
$$\phi_{k}(\epsilon)=\frac{1}{k-1}\bigg[\frac{1}{(a-t)^{k-1}}-\frac{1}{(b-t)^{k-1}}\bigg]+
\Lambda_{k}(\epsilon),$$
where

$$\Lambda_{k}(\epsilon)=\frac{1}{k-1}\bigg[\frac{1}{(\psi(\tau+\epsilon)-\psi(\tau))^{k-1}}
- \frac{1}{(\psi(\tau-\epsilon)-\psi(\tau))^{k-1}}\bigg].$$
Since  $2\leq k\leq m$, $\psi\in C^k [\alpha,\beta]$.
By the fact that $\psi'(\tau)>0$, there exists a function $\theta(\eta)\in C^{k-1}(I)$, where $I=(-\rho,\rho)$ with   $\rho\leq\min\{\beta-\tau,\tau-\alpha\}$,   such that
$$ \psi(\tau+\eta)-\psi(\tau)\equiv\eta\,
\theta(\eta),\quad \theta(0)=\psi'(\tau)>0.$$
Consequently,
$$ \Lambda_{k}(\epsilon)=\epsilon^{-k+1}[M_k(\epsilon)+(-1)^kM_k(-\epsilon)];\quad
M_k(\eta)=\frac{1}{k-1}\frac{1}{[\theta(\eta)]^{k-1}}\in C^{k-1}(I).$$
Expanding in a Taylor series about $\epsilon=0$, we obtain
\begin{align*}\Lambda_{k}(\epsilon)&=\epsilon^{-k+1}
\bigg\{\bigg[\sum^{k-2}_{i=0}c_i\epsilon^i+c^+_{k-1}\epsilon^{k-1}\bigg]
+(-1)^k\bigg[\sum^{k-2}_{i=0}(-1)^ic_i\epsilon^i+(-1)^{k-1}c^-_{k-1}\epsilon^{k-1}\bigg]\bigg\}\\
&=\sum^{k-2}_{i=0}[1+(-1)^{k+i}]c_i\epsilon^{-k+i+1}+ (c^+_{k-1}-c^-_{k-1})\epsilon^0,\end{align*}
where
$$c_i=\frac{M_k^{(i)}(0)}{i!}; \quad c^\pm_{k-1}=\frac{M_k^{(k-1)}(\eta^\pm)}{(k-1)!},\quad
\eta^\pm \ \text{between $0$ and $\pm\epsilon$}.$$
 The  sum $\sum^{k-2}_{i=0}[1+(-1)^{k+i}]c_i\epsilon^{-k+i+1}$ gives a linear combination of the following (odd) powers of $\epsilon$:
$$\epsilon^{-\hat{k}},\epsilon^{-\hat{k}+2},\epsilon^{-\hat{k}+4},\ldots,\epsilon^{-3},
\epsilon^{-1};\quad \hat{k}=\begin{cases}k-2&\text{if $k$ odd}\\ k-1&\text{if $k$ even}\end{cases}.$$
Since each of these powers tends to infinity as $\epsilon\to0+$, we discard them all.
The remaining term, namely,
$$ (c^+_{k-1}-c^-_{k-1})\epsilon^0= \bigg[\frac{M_k^{(k-1)}(\eta^+)}{(k-1)!}-\frac{M_k^{(k-1)}(\eta^-)}{(k-1)!}\bigg]\epsilon^0,$$
 tends to zero as $\epsilon\to0$ because $\lim_{\epsilon\to0}\eta^\pm=0$ and $M_k^{(k-1)}(\eta)$ is continuous in $I$. Thus, we have shown that $\Lambda_k(\epsilon)$ has zero contribution to $\intBar_\alpha^\beta \frac{\psi'(\xi)}{(\psi(\xi)-\psi(\tau))^{k}}\,d\xi$.
 Therefore, we have
\be \label{eqg3} \intBar_\alpha^\beta \frac{\psi'(\xi)}{(\psi(\xi)-\psi(\tau))^{k}}\,d\xi=
\frac{1}{k-1}\bigg[\frac{1}{(a-t)^{k-1}}-\frac{1}{(b-t)^{k-1}}\bigg],\quad k=2,3,\ldots,m,\ee
independent of $\psi(\xi)$.

Substituting \eqref{eqg1},  \eqref{eqg2}, and \eqref{eqg3} in \eqref{eqs2-1}, we obtain
\begin{align*} \intBar_\alpha^\beta  {f}(\psi(\xi))\,\psi'(\xi)\,d\xi&=
\int_a^b  w(x)\,dx + \frac{g^{(m-1)}(t)}{(m-1)!}\log\bigg|\frac{b-t}{a-t}\bigg|\\
&+  \sum^{m-2}_{i=0}\frac{g^{(i)}(t)}{i!}\frac{1}{m-i-1}
\bigg[\frac{1}{(a-t)^{m-i-1}}-\frac{1}{(b-t)^{m-i-1}}\bigg],
\end{align*}
independent of $\psi(\xi)$.
This completes the proof.
\end{proof}

Theorem \ref{thpsi} continues to hold when $f(x)=g(x)/\big|x-t\big|^s$, if $s\geq1$ and $s$ is not an  integer. It also holds when $s$ is an even integer since $\big|x-t\big|^s=(x-t)^s$ in this case, which we have already covered in Theorem \ref{thpsi}. It does {\em not} hold when $s\geq1$ and $s$ is an  odd integer,  however. These facts are the subject of the next two theorems, which we include for completeness.
These theorems can be proved using the technique employed  in proving Theorem \ref{thpsi}. We leave the proofs to the interested reader.

\begin{theorem}\label{thpsi1p}
Let $s>1$ such that  $s$ is not an integer, let $m=\lceil s \rceil$, and let
\be\label{eqpsi1a}f(x) =\frac{g(x)}{\big|x-t\big|^s}, \quad a<t<b,\quad
g\in C^m(a,b),\quad g(x)\ \text{integrable at $x=a$ and $x=b$}, \ee
and let the variable transformation $x=\psi(\xi)$ be as in Theorem \ref{thpsi}.
Then
\be\label{eqpsi3a}
\intBar_\alpha^\beta {f}(\psi(\xi))\,\psi'(\xi)\,d\xi=\intBar^b_a {f}(x)\,dx
\quad\text{independent of $\psi(\xi)$.}\ee
\end{theorem}

\begin{theorem}\label{thpsi1s}
Let $m\geq1$ be an odd integer and let
\be\label{eqpsi1b}f(x) =\frac{g(x)}{\big|x-t\big|^m}, \quad a<t<b,\quad
g\in C^m(a,b),\quad g(x)\ \text{integrable at $x=a$ and $x=b$}, \ee
and let the variable transformation $x=\psi(\xi)$ be as in Theorem \ref{thpsi}.
Then, in general,
\be\label{eqpsi3b}
\intBar_\alpha^\beta {f}(\psi(\xi))\,\psi'(\xi)\,d\xi\not=\intBar^b_a {f}(x)\,dx.
\ee
For example, when $m=1$, we have
\be \intBar_\alpha^\beta {f}(\psi(\xi))\,\psi'(\xi)\,d\xi=\intBar^b_a f(x)\,dx- 2g(t)\log\psi'(\tau),\ee
where
\be \intBar^b_a f(x)\,dx=\int^b_a \frac{g(x)-g(t)}{\big|x-t\big|}\,dx+g(t)
\log\big|(a-t)(b-t)\big|.\ee
\end{theorem}

\section{Periodization of $F(\xi)$ via variable transformations}\label{se4}
\setcounter{equation}{0} \setcounter{theorem}{0}
\subsection{Preliminaries}
In view of  Theorem \ref{thpsi}, with $\psi(\xi)$ as in \eqref{eqpsi22}, we have
\be\label{eqpsi33} I[f]=\intBar^b_a f(x)\,dx=\intBar_\alpha^\beta F(\xi)\,d\xi=I[F], \quad F(\xi)=f(\psi(\xi))\,\psi'(\xi).\ee
We now aim to choose $\psi(\xi)$ so as to  periodize the transformed integrand $F(\xi)$ in the sense that
$$F^{(i)}(\alpha)= F^{(i)}(\beta),\quad i=0,1,\ldots,q-1,\quad \text{for some integer  $q$}.$$
The easiest way of achieving this goal is  by choosing $\psi(\xi)$ such that
\be \label{eqyy}\psi^{(i)}(\alpha)= \psi^{(i)}(\beta)=0,\quad i=1,2,\ldots,r,\quad \text{for some integer  $r$.}\ee
Provided $r$ is sufficiently large,  this will periodize  $F(\xi)$ in the sense that
\be \label{eqxx} F^{(i)}(\alpha)=0= F^{(i)}(\beta),\quad i=0,1,\ldots,q-1,\quad \text{for some integer  $q$.}\ee
To demonstrate this point, let us look at the following examples.

\begin{example}{\em
In case $f(x)$ is sufficiently differentiable on $[a,b]\setminus\{t\}$, \eqref{eqyy} will force \eqref{eqxx} to be valid with
$q=r$. Let us illustrate this for $r=1,2,3$:
We start with
\begin{align} F&=f(\psi)\psi',\label{equ1}\\
F'&=f(\psi)\psi''+f'(\psi)(\psi')^2,\label{equ2}\\ F''&=f(\psi)\psi'''+3f'(\psi)\psi'\psi''+f''(\psi)(\psi')^3. \label{equ3}
\end{align}
\begin{itemize}
\item For $r=1$, that $F(\alpha)=F(\beta)=0$ is obvious  by \eqref{equ1}  and by $\psi'(\alpha)=\psi'(\beta)=0$ in \eqref{eqyy}.
\item
For  $r=2$, that $F^{(i)}(\alpha)=F^{(i)}(\beta)=0$, $i=0,1,$ is obvious  by \eqref{equ1}--\eqref{equ2}  and by $\psi^{(i)}(\alpha)=\psi^{(i)}(\beta)=0$, $i=1,2,$ in \eqref{eqyy}.
\item
For  $r=3$, that $F^{(i)}(\alpha)=F^{(i)}(\beta)=0,$ $i=0,1,2,$ is obvious  by \eqref{equ1}--\eqref{equ3}  and by $\psi^{(i)}(\alpha)=\psi^{(i)}(\beta)=0$, $i=1,2,3,$ in \eqref{eqyy}.
\end{itemize}
It is now easy to see that, if $\psi(\xi)$ satisfies \eqref{eqyy} with some $r$,
we have that \eqref{eqxx} is valid with $q=r$.\\
Clearly,   $q=\infty$ when $r=\infty$.
} \hfill$\square$
\end{example}
\begin{example}\label{ex42}{\em
In case $f(x)$ is not continuous or differentiable at $x=a$ and/or $x=b$ but is integrable there, we can still use the variable transformation $x=\psi(\xi)$ satisfying \eqref{eqyy} and achieve \eqref{eqxx}, but with some $q<r-1$.  To illustrate this point, let us  consider $f(x)=(x-a)^c u_a(x)=(b-x)^c u_b(x)$, where $-1<c<0$ and $u_a(x)$ and $u_b(x)$ are  such that $u_a(a)\neq0$ and $u_b(b)\neq 0$ and are sufficiently differentiable on $[a,b)\setminus\{t\}$ and $(a,b]\setminus\{t\}$, respectively. [The fact that $c>-1$ guarantees  that $f(x)$ is
integrable at $x=a$ and $x=b$, as required in \eqref{eq3}, even though $f(x)$ and all its derivatives are  unbounded at $x=a$ and $x=b$ when $c<0$.]
Now, assuming that $\psi^{(r+1)}(\alpha)\neq0$ and $\psi^{(r+1)}(\beta)\neq0$,
$$\psi(\xi)-\psi(\alpha)\sim \frac{\psi^{(r+1)}(\alpha)}{(r+1)!}(\xi-\alpha)^{r+1}
\quad\text{and}\quad \psi'(\xi)\sim \frac{\psi^{(r+1)}(\alpha)}{r!}(\xi-\alpha)^{r}\quad
\text{as $\xi\to\alpha+$,}$$
and
$$\psi(\xi)-\psi(\beta)\sim \frac{\psi^{(r+1)}(\beta)}{(r+1)!}(\xi-\beta)^{r+1}
\quad\text{and}\quad \psi'(\xi)\sim \frac{\psi^{(r+1)}(\beta)}{r!}(\xi-\beta)^{r}\quad
\text{as $\xi\to\beta-$}.$$
Therefore, $F(\xi)$ satisfies the asymptotic equalities
$$ F(\xi)\sim \begin{cases} M (\xi-\alpha)^\rho&\text{as $\xi\to\alpha+$} \\
N(\beta-\xi)^\rho&\text{as $\xi\to\beta-$}\end{cases},\quad \text{for some $M,N\neq0$}, \quad \rho=c(r+1)+r.$$
 Thus, $F(\xi)$ satisfies \eqref{eqxx}  with
$q=\lceil \rho\rceil$, and $q\geq1$ provided $r\geq(1-c)/(1+c)$, which can be accomplished by choosing $\psi(\xi)$ appropriately.
 When $c=-1/2$, for example, \eqref{eqxx} holds with
$q=\lceil \tfrac{r-1}{2} \rceil$.\\
Clearly,  $q=\infty$ when $r=\infty$ in this example too.
} \hfill$\square$
\end{example}

\noindent {\bf Remark.}
Before going on, we would like to emphasize that variable transformations as described here will be useful only if $f(x)$ is {\em integrable} at the endpoints $x=a$ and $x=b$. If $f(x)$ has {\em nonintegrable} singularities at $x=a$ or $x=b$, then the transformed integrand $F(\xi)$ has {\em worse} singularities at $\xi=\alpha$ or $\xi=\beta$. We can verify this by letting $c<-1$ in Example \ref{ex42}, which causes $\rho<c$; for example, with $c=-3/2$, we have $\rho=c-r/2$.

In subsection \ref{sse43}, we give examples of $\psi(\xi)$ with  both $r$ finite and $r=\infty$.

\subsection{Consequences of periodization}\label{sse42}
The periodization of $F(\xi)$  as in \eqref{eqxx} has important consequences, which we discuss next.
 With $\T=\beta-\alpha$, let us  denote  the $\T$-periodic extension of $F(\xi)$ by  ${\cal F}(\xi)$. Thus,
\be \label{eqcalf} {\cal F}(\xi)=F(\xi)\quad  \text{if $\xi\in(\alpha,\beta)$}\quad\text{and}\quad
{\cal F}(\xi+k\T)={\cal F}(\xi),\quad k=0,\pm1,\pm2,\ldots. \ee
As a result, for arbitrary $i=0,1,\ldots,$
\be {\cal F}^{(i)}(\xi)=\begin{cases} F^{(i)}(\xi)\quad  &\text{if $\xi\in(\alpha,\beta)$,}\\
F^{(i)}(\xi-\T)\quad  &\text{if $\xi\in(\beta,\beta+\T)$.}\end{cases} \ee
Therefore, with $\epsilon>0$ and small,
\be{\cal F}^{(i)}(\beta-\epsilon)=F^{(i)}(\beta-\epsilon),\quad  {\cal F}^{(i)}(\beta+\epsilon)=F^{(i)}(\alpha+\epsilon),\ee
which, upon letting $\epsilon\to0$, gives

\be  {\cal F}^{(i)}(\beta-)=F^{(i)}(\beta),\quad {\cal F}^{(i)}(\beta+)=F^{(i)}(\alpha),\quad
i=0,1,\ldots, q-1, \ee
which, by \eqref{eqxx}, leads to
\be \label{eqFq} {\cal F}^{(i)}(\beta-)={\cal F}^{(i)}(\beta+)=0,\ \ i=0,1,\ldots,q-1
\ \ \Rightarrow\ \ {\cal F}\in C^{q-1}\ \text{in a neighborhood of $\beta$.}\ee
Thus, ${\cal F}(\xi)$,  the $\T$-periodic extension of $F(\xi)$ that is defined for $\xi\in[\alpha,\beta]$  is in $C^{q-1}(\mathbb{R}_\tau)$, where, analogous to \eqref{eqRt},
\beq \label{eqrtau}
\T=\beta-\alpha,\quad \mathbb{R}_\tau=\mathbb{R}\setminus\{\tau+ k\T\}_{k=-\infty}^\infty.\eeq

\noindent {\bf Remark.}
It is clear from the examples we have given above that $q$ increases with $r$. Because $r$ is at our disposal, we may choose $\psi(\xi)$ such that \eqref{eqyy} is satisfied with $r$  as large as we wish, including $r=\infty$. Thus, we can also make $q$ as large  as we wish, including $q=\infty$, forcing  ${\cal F}(\xi)$ to be as smooth as we wish.

\subsection{Examples of periodizing variable transformations} \label{sse43}
Variable transformations were originally developed and used for enhancing the accuracy of the trapezoidal rule approximations to finite-range  integrals  $\int^b_a\phi(x)\,dx$ defined in the regular sense.\footnote{Here we must  emphasize that, in this work, we are  using  variable transformations for the sole purpose of achieving \eqref{eqxx}.}   There are different types of variable transformations; for surveys of these and their applications in numerical integration, see
 Beckers and  Haegemans \cite{Beckers:1992:TIL}, Elliott \cite{Elliott:1998:STT},
  Monegato and  Scuderi \cite{Monegato:1999:NIF}, and
  Sidi \cite{Sidi:1993:NVT}, \cite{Sidi:2006:ECP}, \cite{Sidi:2007:NCS},  \cite{Sidi:2008:FEC}, Yun \cite{Yun:2001:ETG}, and Yun and Kim \cite{Yun:2003:NST}, for example.
Because we only wish to achieve \eqref{eqyy} in this work,
here we will mention, without going into much detail,  only a few of those that have simple representations.

For all the transformations we mention next, we use the standard intervals  $(a,b)=(0,1)$ and  $(\alpha,\beta)=(0,1)$, and we will denote these by $\hpsi(\xi)$ to emphasize this fact.\footnote{In case $(a,b)\neq(0,1)$,  the variable transformation $\psi:[a,b]\to[0,1]$ is simply
$\psi(\xi)=a+(b-a)\hpsi(\xi)$, hence $\intBar^b_af(x)\,dx=\intBar^1_0F(\xi)\,d\xi$, with
$F(\xi)=f(\psi(\xi))\psi'(\xi)=(b-a)f(a+(b-a)\hpsi(\xi))\hpsi'(\xi)$.}

\begin{align}
&\text{Korobov \cite{Korobov:1963:NTM}:} &&\hpsi(\xi)=\frac{\theta(\xi)}{\theta(1)},\quad \theta(\xi)=\int^\xi_0 [u(1-u)]^{p-1}\,du.\label{eqc1}\\
&\text{Sidi \cite{Sidi:1993:NVT}, \cite{Sidi:2008:FEC}:} &&\hpsi(\xi)=\frac{\theta(\xi)}{\theta(1)},\quad \theta(\xi)=\int^\xi_0 (\sin{\pi u})^{p-1}\,du. \label{eqc2}\\
&\text{Pr\"{o}sdorf and Rathsfeld \cite{Prossdorf:1991:QMS}:} &&\hpsi(\xi)=\frac{\xi^{p}}{\xi^{p}+(1-\xi)^{p}}.\label{eqc3} \\
&\text{Sidi \cite{Sidi:2007:NCS}:}&& \hpsi(\xi)= \frac{\displaystyle\bigg(\sin\frac{\pi \xi}{2}\bigg)^{p}}{\displaystyle\bigg(\sin\frac{\pi \xi}{2}\bigg)^{p}+\bigg(\cos\frac{\pi \xi}{2}\bigg)^{p}}.\label{eqc4} \\
 &\text{Sag and Szekeres \cite{Sag:1964:NEH}:}
&&\hpsi(\xi)=\frac{1}{2}\tanh\bigg(c\bigg(\frac{1}{1-\xi}-\frac{1}{\xi}\bigg)\bigg)+\frac{1}{2},\quad c>0.\label{eqc5}
\end{align}
Of course, by Theorem \ref{thpsi}, with each of the transformations in \eqref{eqc1}--\eqref{eqc5}, we have
$$ \intBar^1_0 f(x)\,dx=\intBar^1_0f(\hpsi(\xi))\,\hpsi'(\xi)\,d\xi, \quad
f(x)=\frac{g(x)}{(x-t)^m}\,dx,\quad 0<t<1.$$

With integer $p$, for all the transformations in \eqref{eqc1}--\eqref{eqc4}, we have

$$ \hpsi^{(i)}(0)=\hpsi^{(i)}(1)=0,\quad i=1,\ldots, r,$$
with (i)\,$r=p-1$  in \eqref{eqc1}--\eqref{eqc4},
and (ii)\,$r=\infty$ in \eqref{eqc5}.
\medskip

\noindent{\bf Remarks.}
\begin{enumerate}
\item
Note that $p$ in \eqref{eqc1}--\eqref{eqc4} does not have to be chosen as an integer. It can be chosen as an integer or otherwise so as to optimize the quality of  the quadrature approximations for regular integrals.
 \item
 All five variable transformations we just mentioned satisfy
 $$\hpsi(1-\xi)=1-\hpsi(\xi)\quad \text{and}\quad \hpsi'(1-\xi)=\hpsi'(\xi),\quad
 \xi\in[0,1].$$
 Clearly, $\hpsi'(\xi)$ are symmetric with respect to $\xi=1/2$, that is,
 $$\hpsi(\tfrac{1}{2})=\tfrac{1}{2};\quad \hpsi'(\tfrac{1}{2})=0,\quad \hpsi'(\tfrac{1}{2}+\epsilon)=\hpsi'(\tfrac{1}{2}-\epsilon).$$
\item
For $\hpsi(\xi)$ in \eqref{eqc3}--\eqref{eqc5}, we can obtain    $\tau$ from $\psi(\tau)=t$ analytically. In all three cases, $\tau$ is the solution to the equation
$$ \frac{\rho(\tau)}{\rho(\tau)+1}=t \quad\Rightarrow\quad \rho(\tau)=\frac{t}{1-t},$$ where
\begin{align}
\text{for \eqref{eqc3}:}\quad &\rho(\tau)=\bigg(\frac{\tau}{1-\tau}\bigg)^p, \\
\text{for \eqref{eqc4}:}\quad & \rho(\tau)=\bigg(\tan\frac{\pi\tau}{2}\bigg)^p, \\
\text{for \eqref{eqc5}:}\quad &\rho(\tau)=\exp\bigg[2c\bigg(\frac{1}{1-\tau}-\frac{1}{\tau}\bigg)\bigg].
\end{align}
Then we have the following:
\begin{align}
\text{For \eqref{eqc3}:}\quad \tau&=\frac{t^{1/p}}{t^{1/p}+(1-t)^{1/p}}. \label{eqtau1}\\
\text{For \eqref{eqc4}:}\quad \tau&=\frac{2}{\pi}\tan^{-1}\lambda;\quad \lambda=\bigg(\frac{t}{1-t}\bigg)^{1/p}.\\
\text{For \eqref{eqc5}:}\quad \tau&=\begin{cases}\displaystyle\frac{\sqrt{\lambda^2+4}+\lambda-2}{2\lambda}\quad  \text{if $\lambda>0$}\\ \vspace{-3mm}\\ \displaystyle\frac{2}{\sqrt{\lambda^2+4}-\lambda+2}\quad \text{if $\lambda\leq0$}\end{cases};
\quad
\lambda=\frac{1}{2c}\log\bigg(\frac{t}{1-t}\bigg).
\end{align}
\end{enumerate}

\section{PVTSI$^{(m)}$: Development of numerical quadrature formulas via periodization}\label{se5}
\setcounter{equation}{0} \setcounter{theorem}{0}
We go back to the HFP integrals $\intBar^b_af(x)\,dx$  described in Section \ref{se1}, assuming that $f(x)$ is as in \eqref{eq3};  therefore,  $f\in C^\infty((a,b)\setminus\{t\})$ since   $g\in C^\infty(a,b)$. Let us periodize $f(x)$ via a suitable variable transformation $x=\psi(\xi)$, where
\be \label{eqpsp}\psi\in C[\alpha,\beta],\ \psi\in C^\infty(\alpha,\beta);\ \ \psi(\alpha)=a,\ \psi(\beta)=b;\ \
  \psi'(\xi)>0\ \text{for $\xi\in(\alpha,\beta)$},\ee
as described in the preceding section, such that the transformed integrand
$F(\xi)=f(\psi(\xi))\,\psi'(\xi)$ satisfies \eqref{eqxx},
which is possible by a judicious choice of $\psi(\xi)$, as we have already seen.
Then we have $I[f]=I[F]$ by Theorem \ref{thpsi}, and

\be \label{eqI[f]}I[F]=\intBar^\beta_\alpha F(\xi)\,d\xi=\intBar^{\tau+\T}_\tau {\cal F}(\xi)\,d\xi,\quad \T=\beta-\alpha,\ee
 the function ${\cal F}(\xi)$ being    the $\T$-periodic extension of
$F(\xi)$ introduced in subsection \ref{sse42}.  Our aim is to develop numerical quadrature formulas specifically for the integral $\intBar^{\tau+\T}_\tau {\cal F}(\xi)\,d\xi$ to compute $I[F]$ as given in \eqref{eqI[f]}. We will achieve this via Theorem \ref{thA} that is proved in the appendix. For this, we need to study in detail the analytical properties of ${\cal F}(\xi)$ in the interval $(\tau,\tau+\T)$.

\begin{enumerate}

\item
First, note that both endpoints $\xi=\tau$ and $\xi=\tau+\T$ of the integration interval $J=(\tau,\tau+\T)$ are points of singularity of ${\cal F}(\xi)$ and ${\cal F}\in C^\infty(J\setminus\{\beta\})$.
Next, at  $\xi=\beta$, which is in the interior of $J$, the function
${\cal F}(\xi)$ is continuous and has $q-1$ continuous derivatives, and
${\cal F}^{(i)}(\beta)=0,$ $i=0,1,\ldots,q-1,$ by \eqref{eqFq}. We now assume, without loss of generality,  that ${\cal F}^{(q)}(\xi)$ is absolutely integrable through $\xi=\beta$, which can be achieved by choosing $\psi(\xi)$ appropriately. To have a visual idea about what we have just explained, see the Figures \ref{fig1}, \ref{fig2}, and \ref{fig3}.

 \item
 Let us  choose   $\sigma$ and $\rho$ such that
$$\tau<\sigma<\beta<\rho<\tau+\T,$$
and let
$$J_1=(\tau,\sigma)\cup(\rho,\tau+\T),\quad J_2=[\sigma,\rho].$$
Then ${\cal F}\in C^\infty(J_1)$ and ${\cal F}\in C^{q-1}(J_2)$,
${\cal F}^{(q)}(\xi)$ being absolutely integrable in $J_2$.

\item
We are now interested in determining the asymptotic expansions of ${\cal F}(\xi)$  as $\xi\to\tau+$ and as $\xi\to(\tau+\T)-$, recalling from \eqref{eqpsi43} that $F(\xi)$ can be expressed as $F(\xi)=G(\xi)/(\xi-\tau)^m$ with $G\in C^\infty(\alpha,\beta)$.

\begin{itemize}
\item
First, we have that
$${\cal F}(\xi)=F(\xi)=\frac{G(\xi)}{(\xi-\tau)^m}\quad \text{if $\xi\in(\alpha,\beta)$}.$$
Since $\tau\in(\alpha,\beta)$, expanding $G(\xi)$ in a Taylor series about $\xi=\tau$, we have that
$${\cal F}(\xi)\sim \sum^\infty_{i=0}\frac{G^{(i)}(\tau)}{i!}(\xi-\tau)^{i-m}\quad \text{as $\xi\to\tau$},$$ which we write in the form
\be\label{eqF111} {\cal F}(\xi)\sim \frac{G^{(m-1)}(\tau)}{(m-1)!} (\xi-\tau)^{-1}+ \sum^\infty_{\substack{i=0\\ i\neq m-1}}\frac{G^{(i)}(\tau)}{i!}(\xi-\tau)^{i-m}\quad \text{as $\xi\to\tau$}.\ee

\item
Next, by $\T$-periodicity of ${\cal F}(\xi)$,
 we have that
$${\cal F}(\xi)=F(\xi-\T)=\frac{G(\xi-\T)}{(\xi-\T-\tau)^m}
\quad \text{if $\xi\in (\beta, \beta+\T)$.}$$
Expanding $G(\xi-\T)$ in a Taylor series about $\xi=\tau+\T$, which is in
 $(\beta, \beta+\T)$, we have the asymptotic expansion

$$ {\cal F}(\xi)\sim \sum^\infty_{i=0}\frac{G^{(i)}(\tau)}{i!}(\xi-\tau-\T)^{i-m}
\quad\text{as $\xi\to \tau+\T$}, $$ which we write in the form
\begin{multline}\label{eqF222}{\cal F}(\xi)\sim -\frac{G^{(m-1)}(\tau)}{(m-1)!} (\tau+\T-\xi)^{-1}\\
+\sum^\infty_{\substack{i=0\\ i\neq m-1}}(-1)^{i-m}\frac{G^{(i)}(\tau)}{i!}(\tau+\T-\xi)^{i-m}
\quad\text{as $\xi\to \tau+\T$}.\end{multline}
\end{itemize}

\item
Thus, Theorem \ref{thA} applies to $\intBar^{\tau+\T}_\tau {\cal F}(\xi)\,d\xi$  with $h={\cal T}/n$, and we have, as $h\to0$,
\beq \label{eqasd} h\sum^{n-1}_{j=1} {\cal F}(\tau+jh)\sim I[F]+R_{q}(h)+\sum^\infty_{\substack{i=0\\ i\neq m-1}}[1+(-1)^{i-m}]\,\frac{G^{(i)}(\tau)}{i!}\zeta(-i+m)h^{i-m+1},\eeq
where $\zeta(z)$ is the Riemann Zeta function and
\beq \label{eqasd1}
R_{q}(h)=O(h^{q})\quad\text{as $h\to0$}. \eeq
Observe that the contributions of the terms involving $(\xi-\tau)^{-1}$ and $(\tau+\T-\xi)^{-1}$ that appear in  \eqref{eqF111} and \eqref{eqF222} cancel each other.
\end{enumerate}

Now the infinite sum in \eqref{eqasd} contains only those terms with  $i-m$ an even integer, which   involve  the zeta function $\zeta(2s)$, $s$ being an integer,  positive, negative, or zero.
Invoking  the fact that $\zeta(-2k)=0$ for $k=1,2,\ldots,$  we see that this (infinite) sum reduces further to a {\em finite} sum.  We summarize the end result in the  following theorem:

\begin{theorem}\label{thwww}
Depending on whether $m$ is even or  odd, \eqref{eqasd} assumes the following forms:\\
\noindent{1. For $m=2r$, $r=1,2,\ldots,$}
\be \label{eqmeven}h\sum^{n-1}_{j=1} {\cal F}(\tau+jh)=I[F]+
2\sum^r_{i=0}\frac{G^{(2i)}(\tau)}{(2i)!}\,\zeta(2r-2i)\,h^{-2r+2i+1}+O(h^{q})
\quad \text{as $h\to0$}.\ee
\noindent{2. For $m=2r+1$, $r=0,1,\ldots,$}
\be \label{eqmodd}h\sum^{n-1}_{j=1} {\cal F}(\tau+jh)=I[F]+
2\sum^r_{i=0}\frac{G^{(2i+1)}(\tau)}{(2i+1)!}\,\zeta(2r-2i)\,h^{-2r+2i+1}+O(h^{q})
\quad \text{as $h\to0$}.\ee
Clearly, $q$ depends only on (i)\,$g(x)$ at $x=a$ and $x=b$ and (ii)\,$\psi(\xi)$, and is independent of~$m$.
\end{theorem}

In view of \eqref{eqmeven}--\eqref{eqmodd}, we define our PVTSI$^{(m)}$ numerical quadrature  formulas $\widehat{T}^{(s)}_{m,n}[{\cal F}]$ for $I[F]$ precisely as those in \cite{Sidi:2019:SSI-P1}, which we have summarized in  Section \ref{se2}:\\
\noindent{1. For $m=2r$, $r=1,2, \ldots,$}

\be \label{eqTmeven0} \widehat{T}^{(0)}_{2r,n}[{\cal F}]=h\sum^{n-1}_{j=1} {\cal F}(\tau+jh)
-2\sum^r_{i=0}\frac{G^{(2i)}(\tau)}{(2i)!}\,\zeta(2r-2i)\,h^{-2r+2i+1}.
\ee
\noindent{2. For $m=2r+1$, $r=0,1, \ldots,$}
\be \label{eqTmodd0} \widehat{T}^{(0)}_{2r+1,n}[{\cal F}]=h\sum^{n-1}_{j=1} {\cal F}(\tau+jh)
-2\sum^r_{i=0}\frac{G^{(2i+1)}(\tau)}{(2i+1)!}\,\zeta(2r-2i)\,h^{-2r+2i+1}.\ee

From these, we obtain  the rest of the formulas $\widehat{T}^{(s)}_{m,n}[{\cal F}]$ with $s=1,2,\ldots$ precisely as the $\widehat{T}^{(s)}_{m,n}[f]$ described in Section \ref{se2}.  For example, with
$h=\T/n,$ for $m=1,2,3$, we have the following numerical quadrature formulas:

\begin{enumerate}
\item The case  $m=1$:
\begin{subequations}
\begin{align} \widehat{T}^{(0)}_{1,n}[{\cal F}]&=h\sum^{n-1}_{j=1}{\cal F}(\tau+jh)+G'(\tau)\,h,  \label{eqTT10}  \\
 \widehat{T}^{(1)}_{1,n}[{\cal F}]&=h\sum^n_{j=1}{\cal F}(\tau+jh-h/2) \label{eqTT11}.
 \end{align}
 \end{subequations}
 \item The case  $m=2$:
 \begin{subequations}
\begin{align} \widehat{T}^{(0)}_{2,n}[{\cal F}]&=h\sum^{n-1}_{j=1}{\cal F}(\tau+jh)-\frac{\pi^2}{3}\,G(\tau)\,h^{-1}
   +\frac{1}{2}\,G''(\tau)\,h, \label{eqTT20}\\
 \widehat{T}^{(1)}_{2,n}[{\cal F}]&=h\sum^n_{j=1}{\cal F}(\tau+jh-h/2)
 -\pi^2 G(\tau)h^{-1}, \label{eqTT21}  \\
\widehat{T}^{(2)}_{2,n}[{\cal F}]&=2h\sum^n_{j=1}{\cal F}(\tau+jh-h/2)-
\frac{h}{2}\sum^{2n}_{j=1}{\cal F}(\tau+jh/2-h/4) \label{eqTT22}.
\end{align}
 \end{subequations}
\item  The case  $m=3$:
\begin{subequations}
\begin{align} \widehat{T}^{(0)}_{3,n}[{\cal F}]&=h\sum^{n-1}_{j=1}{\cal F}(\tau+jh)-\frac{\pi^2}{3}\,G'(\tau)\,h^{-1}
   +\frac{1}{6}\,G'''(\tau)\,h, \label{eqTT30}\\
 \widehat{T}^{(1)}_{3,n}[{\cal F}]&=h\sum^n_{j=1}{\cal F}(\tau+jh-h/2)-\pi^2\,G'(\tau)\,h^{-1}, \label{eqTT31}\\
 \widehat{T}^{(2)}_{3,n}[{\cal F}]&=2h\sum^n_{j=1}{\cal F}(\tau+jh-h/2)-
\frac{h}{2}\sum^{2n}_{j=1}{\cal F}(\tau+jh/2-h/4) \label{eqTT32}.
\end{align}
 \end{subequations}
 \item
{\em The case $m=4$}:
\begin{subequations}
\begin{align}
\widehat{T}^{(0)}_{4,n}[{\cal F}]&=h\sum^{n-1}_{j=1}{\cal F}(\tau+jh)-\frac{\pi^4}{45}G(\tau)h^{-3}
-\frac{\pi^2}{6}G''(\tau)h^{-1}+\frac{1}{24}G^{(4)}(\tau)h \label{eqT40} \\
\widehat{T}^{(1)}_{4,n}[{\cal F}]&=h\sum^{n}_{j=1}{\cal F}(\tau+jh-h/2)-\frac{\pi^4}{3}G(\tau)h^{-3}
-\frac{\pi^2}{2}G''(\tau)h^{-1} \label{eqT41} \\
\widehat{T}^{(2)}_{4,n}[{\cal F}]&=2h\sum^{n}_{j=1}{\cal F}(\tau+jh-h/2)
-\frac{h}{2}\sum^{2n}_{j=1}{\cal F}(\tau+jh/2-h/4)+2\pi^4G(\tau)h^{-3} \label{eqT42}  \\
\widehat{T}^{(3)}_{4,n}[{\cal F}]&=\frac{16h}{7}\sum^{n}_{j=1}{\cal F}(\tau+jh-h/2)
-\frac{5h}{7}\sum^{2n}_{j=1}{\cal F}(\tau+jh/2-h/4)\notag \\
&\hspace{4.5cm}+\frac{h}{28}\sum^{4n}_{j=1}{\cal F}(\tau+jh/4-h/8) \label{eqT43}
\end{align}
\end{subequations}
\end{enumerate}
 Before, we go on, we wish to emphasize that no limitations are put on $\tau\in(\alpha,\beta)$ in these formulas.    Therefore, no limitations are put on $t\in(a,b)$, either.

Concerning all of these formulas, we have the following convergence theorem that is analogous to Theorem \ref{thw1t}:
\begin{theorem}\label{thw2} Under the conditions imposed on $f(x)$,  $F(\xi)$, and ${\cal F}(\xi)$, there holds $\lim_{n\to\infty}\widehat{T}^{(s)}_{m,n}[{\cal F}]=I[F]$ for all $m$ and $s$. Actually, there holds

\be \widehat{T}^{(s)}_{m,n}[{\cal F}]-I[F]=O(n^{-q})\quad\text{as $n\to\infty$,\quad for each $m$ and $s$}.\ee
Clearly, $q$ depends only on (i)\,$g(x)$ at $x=a$ and $x=b$ and (ii)\,$\psi(\xi)$, and is independent of~$m$.
\end{theorem}

\noindent{\bf Remarks.}
\begin{enumerate}
\item \label{re52}
Let us  recall the first remark at the end of subsection \ref{sse42} that says that we can make $q$ in \eqref{eqxx} as large as we wish by choosing $\psi(\xi)$ appropriately. From this and from Theorem \ref{thw2}, it is clear that we can  improve the accuracy of the $\widehat{T}^{(s)}_{m,n}[{\cal F}]$  at will by choosing $\psi(\xi)$ such that  $r$ in \eqref{eqyy} is sufficiently large to make $q$ as large as we wish.
\item
Note that because  $\alpha<\tau<\beta$, some of the abscissas in each of the quadrature formulas above are in $(\alpha,\beta)$, while others  are necessarily in $(\beta,\beta+\T)$.
We can invoke the $\T$-periodicity of ${\cal F}(\xi)$ for those abscissas in
$(\beta,\beta+\T)$. Thus, in all the formulas
$\widehat{T}^{(0)}_{m,n}[{\cal F}]$, (i)\,if $\tau+jh\leq\beta$, then ${\cal F}(\tau+jh)=F(\tau+jh)$,
while (ii)\,if $\tau+jh>\beta$,
then ${\cal F}(\tau+jh)=F(\tau+jh-\T)=F(\tau-(n-j)h)$ since $\tau-(n-j)h\in[\alpha,\tau)$.
\item
Clearly, we need  $G^{(i)}(\tau)$, $i=0,1,2,\ldots,$ in the
quadrature formulas $\widehat{T}^{(s)}_{m,n}[{\cal F}]$. We recall that  $G(\tau)$
and $G'(\tau)$ are  given in \eqref{eqpsi44}--\eqref{eqpsi45}. All of the $G^{(i)}(\tau)$ can be obtained by differentiating  $G(\xi)$  in \eqref{eqpsi43}  and letting $\xi\to\tau$. For this, it is clear that  we also need the derivatives with respect to $\xi$ of $Q(\xi)=\psi[\xi,\tau]$, evaluated at $\xi=\tau$.
Expanding $\psi(\xi)$ in a Taylor series about $\xi=\tau$, it is readily seen that
$$ Q^{(k)}(\tau)=\frac{\psi^{(k+1)}(\tau)}{k+1}, \quad k=0,1,\ldots.$$
\end{enumerate}

\section{Numerical examples with PVTSI$^{(m)}$ quadrature formulas}\label{se6}
\setcounter{equation}{0} \setcounter{theorem}{0}
We have applied the PVTSI$^{(m)}$ formulas, with the variable transformations in \eqref{eqc3}--\eqref{eqc5}, to several HFP integrals with $m=1,2,3.$ The numerical results obtained lead us to conclude that they achieve high accuracies  in all cases.

In all these examples,  $[a,b]=[0,1]$ and $[\alpha,\beta]=[0,1]$ and we present those results obtained by using the variable transformation $\hpsi(\xi)=\xi^p/[\xi^p+(1-\xi)^p]$ with $p=5$, $p=10$,  and $p=15$. (All three transformations in \eqref{eqc3}--\eqref{eqc5} seem to produce very similar numerical results.)
We have carried out all our computations in quadruple-precision arithmetic (approximately 34 decimal digits).

Below, we use the notation
$$ I_m(t)=\intBar^b_a \frac{g(x)}{(x-t)^m}\,dx, \quad m=1,2 \ldots.$$

We first  treat three HFP integrals
involving the Chebyshev polynomials of the first and second kinds, namely,
$T_k(z)$ and $U_k(z)$, respectively,
 as examples.  The integrands have square-root singularities at the endpoints in all cases.
In all the three examples $g(x)$ is finite at the endpoints $x=a$ and $x=b$, but all its derivatives are unbounded there. (Note that the first two examples were also treated by
Choi, Kim, and Yun  \cite{Choi:2004:ISB}.)

\begin{example}\label{ex1}{\em With $m=1$:
$$ I_1(t)=\intbar^1_0\sqrt{x(1-x)}\,\frac{U_k(2x-1)}{x-t}\,dx=-\frac{\pi}{2}\, T_{k+1}(2t-1).$$
We have computed $I_1(t)$ with $k=4$ and  for $t=0.3$.\\
The exact value of the integral is $I_1(0.3)=1.38833262547440142794141136393888$.
The results of the computation are given in Table \ref{table1}.
\begin{table}[ht]
$$ \begin{array}{||r||c|c||c|c||c|c||}
\hline
k&E^{(0)}_{1,2^k},\,p=5&E^{(1)}_{1,2^k},\,p=5&E^{(0)}_{1,2^k},\,p=10&E^{(1)}_{1,2^k},\,p=10&
E^{(0)}_{1,2^k},\,p=15&E^{(1)}_{1,2^k},\,p=15\\
\hline\hline
     1& 6.021D+00& 9.420D-01& 1.295D+01& 1.000D+00& 1.989D+01& 1.000D+00\\
     2& 2.539D+00& 1.586D+00& 5.973D+00& 9.413D-01& 9.445D+00& 9.974D-01\\
     3& 4.769D-01& 4.916D-01& 2.516D+00& 1.587D+00& 4.224D+00& 9.209D-01\\
     4& 7.352D-03& 7.352D-03& 4.641D-01& 4.807D-01& 1.652D+00& 1.654D+00\\
     5& 1.505D-08& 1.505D-08& 8.315D-03& 8.315D-03& 1.401D-03& 1.237D-03\\
     6& 6.128D-15& 6.137D-15& 2.398D-08& 2.398D-08& 8.164D-05& 8.164D-05\\
     7& 4.636D-18& 4.920D-18& 1.529D-20& 1.529D-20& 3.535D-12& 3.535D-12\\
     8& 1.420D-19& 1.397D-19& 4.162D-34& 1.387D-34& 3.760D-29& 3.758D-29\\
     9& 1.131D-21& 1.142D-21& 2.774D-34& 2.358D-33& 1.110D-32& 3.510D-32\\
    10& 5.575D-24& 5.624D-24& 1.249D-33& 5.965D-33& 2.303D-32& 7.047D-32\\
\hline
\end{array}$$
\caption{\label{table1} Relative errors in the approximations $\widehat{T}^{(0)}_{1,n}[{\cal F}]$
and $\widehat{T}^{(1)}_{1,n}[{\cal F}]$  for $t=0.3$ in Example \ref{ex1}.
Here $E^{(s)}_{1,n}=|\widehat{T}^{(s)}_{1,n}[{\cal F}] -I_1(0.3)|/|I_1(0.3)|$.}
\end{table}
}
\end{example}

\begin{example}\label{ex2}{\em With $m=2$:
$$I_2(t)=\intBar^1_0\sqrt{x(1-x)}\,\frac{U_k(2x-1)}{(x-t)^2}\,dx=-\pi(k+1)U_k(2t-1).$$
We have computed $I_2(t)$ with $k=4$ and for $t=0.3.$ \\
The exact value of the integral is $I_2(0.3)=8.01734445196115234455666591412929$.
The results of the computation are given in Table \ref{table2}.

\begin{table}[ht]
$$ \begin{array}{||r|| c|c||c|c||c|c||}
\hline
k&E^{(1)}_{2,2^k},\,p=5&E^{(2)}_{2,2^k},\,p=5&E^{(1)}_{2,2^k},\,p=10&E^{(2)}_{2,2^k},\,p=10&
E^{(1)}_{2,2^k},\,p=15&E^{(2)}_{2,2^k},\,p=15\\
\hline\hline
    1& 8.428D-01& 9.769D-01& 9.316D-01& 1.027D+00& 9.543D-01& 1.001D+00\\
    2& 7.087D-01& 1.137D+00& 8.357D-01& 9.643D-01& 9.076D-01& 1.084D+00\\
    3& 2.807D-01& 5.613D-01& 7.072D-01& 1.126D+00& 7.310D-01& 7.883D-01\\
    4& 2.132D-04& 4.263D-04& 2.879D-01& 5.756D-01& 6.737D-01& 1.284D+00\\
    5& 1.155D-09& 2.311D-09& 2.176D-04& 4.352D-04& 6.352D-02& 1.271D-01\\
    6& 2.748D-15& 5.486D-15& 2.040D-09& 4.080D-09& 2.245D-05& 4.491D-05\\
    7& 9.955D-18& 1.990D-17& 1.149D-22& 2.298D-22& 1.148D-13& 2.296D-13\\
    8& 1.494D-20& 2.925D-20& 5.765D-34& 1.286D-31& 1.192D-30& 2.469D-30\\
    9& 6.307D-22& 1.263D-21& 8.244D-32& 9.280D-31& 8.360D-32& 6.822D-32\\
   10& 1.485D-24& 2.990D-24& 8.552D-32& 2.994D-30& 9.897D-32& 6.726D-33\\
\hline
\end{array}$$
\caption{\label{table2} Relative errors in the approximations $\widehat{T}^{(1)}_{2,n}[{\cal F}]$
and $\widehat{T}^{(2)}_{2,n}[{\cal F}]$  for $t=0.3$ in Example \ref{ex2}.
Here $E^{(s)}_{2,n}=|\widehat{T}^{(s)}_{2,n}[{\cal F}] -I_2(0.3)|/|I_2(0.3)|$.}
\end{table}
}
\end{example}

\begin{example}\label{ex3}{\em With $m=3$:
$$ I_3(t)=\intBar^1_0\sqrt{x(1-x)}\,\frac{U_k(2x-1)}{(x-t)^3}\,dx=-\pi(k+1)U'_k(2t-1).$$
Here $U_k'(z)=\frac{d}{dz}U_k(z)$. We have computed $I_3(t)$ with $k=4$ and for $t=0.3.$\\ The exact value of the integral is $I_3(0.3)=-86.4566298267911099224919459078519$.
The results of the computation are given in Table \ref{table3}.

\begin{table}[ht]
$$ \begin{array}{||r|| c|c||c|c||c|c||}
\hline
k&E^{(1)}_{3,2^k},\,p=5&E^{(2)}_{3,2^k},\,p=5&E^{(1)}_{3,2^k},\,p=10&E^{(2)}_{3,2^k},\,p=10&
E^{(1)}_{3,2^k},\,p=15&E^{(2)}_{3,2^k},\,p=15\\
\hline\hline
    1& 7.291D-01& 1.069D+00& 8.632D-01& 1.000D+00& 9.087D-01& 1.000D+00\\
    2& 3.889D-01& 7.611D-01& 7.262D-01& 1.068D+00& 8.174D-01& 1.009D+00\\
    3& 1.665D-02& 3.339D-02& 3.840D-01& 7.517D-01& 6.263D-01& 1.078D+00\\
    4& 8.650D-05& 1.730D-04& 1.631D-02& 3.272D-02& 1.748D-01& 3.510D-01\\
    5& 3.819D-11& 7.638D-11& 1.076D-04& 2.152D-04& 1.410D-03& 2.820D-03\\
    6& 9.547D-16& 1.908D-15& 6.819D-11& 1.364D-10& 4.450D-07& 8.901D-07\\
    7& 1.382D-18& 2.751D-18& 1.172D-23& 2.345D-23& 6.046D-15& 1.209D-14\\
    8& 1.332D-20& 2.645D-20& 8.009D-31& 1.031D-30& 1.123D-30& 6.811D-30\\
    9& 1.988D-22& 3.982D-22& 5.260D-30& 1.194D-29& 9.029D-30& 5.592D-29\\
   10& 6.917D-25& 1.394D-24& 3.938D-29& 1.119D-28& 7.368D-29& 4.094D-28\\
\hline
\end{array}$$
\caption{\label{table3} Relative errors in the approximations $\widehat{T}^{(1)}_{3,n}[{\cal F}]$
and $\widehat{T}^{(2)}_{3,n}[{\cal F}]$  for $t=0.3$ in Example \ref{ex3}.
Here $E^{(s)}_{3,n}=|\widehat{T}^{(s)}_{3,n}[{\cal F}] -I_3(0.3)|/|I_3(0.3)|$.  }
\end{table}
}
\end{example}

In the  next three examples,  we have obtained the $I_m(t)$ via iteration of the known relation (see Kaya and Erdogan \cite{Kaya:1987:SIE}, for example)
$$ \intBar^b_a\frac{g(x)}{(x-t)^{k+1}}\,dx=\frac{1}{k}\frac{d}{dt}
\intBar^b_a\frac{g(x)}{(x-t)^k}\,dx,\quad k=1,2,\ldots.$$
Thus, starting with
$$ I_1(t)=M(t)+g(t)H(t); \quad M(t)=\intbar^b_a\frac{g(x)-g(t)}{x-t}\,dx,\quad H(t)=\log\frac{b-t}{t-a},$$ we have
$$ I_2(t)=M'(t)+g(t)H'(t)+g'(t)H(t),$$
$$I_3(t)=\frac{1}{2}\big[M''(t)+g(t)H''(t)+2g'(t)H'(t)+g''(t)H(t)\big],$$
and so on.

\begin{example}\label{ex4}{\em With $m=1$:
$$ I_1(t)=\intbar^1_0\frac{1+x-x^2}{x-t}\,dx=\frac{1}{2}-t+(1+t-t^2)\log\frac{1-t}{t}.$$
We have computed $I_1(t)$   for $t=0.3$.\\
The exact value of the integral is $I_1(0.3)=1.22523041106851637258923008288999$.
The results of the computation are given in Table \ref{table4}.

\begin{table}[ht]
$$ \begin{array}{||r|| c|c||c|c||c|c||}
\hline
k&E^{(0)}_{1,2^k},\,p=5&E^{(1)}_{1,2^k},\,p=5&E^{(0)}_{1,2^k},\,p=10&E^{(1)}_{1,2^k},\,p=10&
E^{(0)}_{1,2^k},\,p=15&E^{(1)}_{1,2^k},\,p=15\\
\hline\hline
 1& 1.596D+00& 8.813D-01& 4.289D+00& 9.989D-01& 6.961D+00& 1.000D+00\\
 2& 3.573D-01& 3.343D-01& 1.645D+00& 8.913D-01& 2.980D+00& 9.871D-01\\
 3& 1.148D-02& 1.151D-02& 3.769D-01& 3.515D-01& 9.966D-01& 7.188D-01\\
 4& 1.269D-05& 1.269D-05& 1.269D-02& 1.273D-02& 1.389D-01& 1.381D-01\\
 5& 2.058D-09& 2.226D-09& 1.674D-05& 1.674D-05& 3.762D-04& 3.764D-04\\
 6& 8.396D-11& 8.672D-11& 6.311D-12& 6.311D-12& 8.610D-08& 8.610D-08\\
 7& 1.384D-12& 1.544D-12& 7.592D-24& 7.778D-24& 3.897D-16& 3.897D-16\\
 8& 8.010D-14& 7.609D-14& 9.289D-26& 9.300D-26& 3.348D-32& 1.030D-31\\
 9& 2.005D-15& 2.169D-15& 5.720D-29& 5.728D-29& 6.319D-32& 2.020D-31\\
 10& 8.210D-17& 8.384D-17& 4.841D-32& 2.358D-33& 1.383D-31& 4.139D-31\\
  \hline
\end{array}$$
\caption{\label{table4} Relative errors in the approximations $\widehat{T}^{(0)}_{1,n}[{\cal F}]$
and $\widehat{T}^{(1)}_{1,n}[{\cal F}]$  for $t=0.3$ in Example \ref{ex4}.
Here $E^{(s)}_{1,n}=|\widehat{T}^{(s)}_{1,n}[{\cal F}] -I_1(0.3)|/|I_1(0.3)|$.  }
\end{table}
}
\end{example}

\begin{example}\label{ex5}{\em With $m=2$:
$$ I_2(t)=\intBar^1_0\frac{1+x-x^2}{(x-t)^2}\,dx=-1-\frac{1+t-t^2}{t(1-t)}
+(1-2t)\log\frac{1-t}{t}.$$
We have computed $I_2(t)$   for $t=0.3$.\\
The exact value of the integral is $I_2(0.3)=-6.42298561774988045927786175929650.$
The results of the computation are given in Table \ref{table5}.
\begin{table}[ht]
$$ \begin{array}{||r|| c|c||c|c||c|c||}
\hline
k&E^{(1)}_{2,2^k},\,p=5&E^{(2)}_{2,2^k},\,p=5&E^{(1)}_{2,2^k},\,p=10&E^{(2)}_{2,2^k},\,p=10&
E^{(1)}_{2,2^k},\,p=15&E^{(2)}_{2,2^k},\,p=15\\
\hline\hline
     1& 2.009D-01& 3.877D-01& 5.588D-01& 9.044D-01& 7.051D-01& 9.889D-01\\
     2& 1.406D-02& 2.911D-02& 2.131D-01& 4.106D-01& 4.213D-01& 7.419D-01\\
     3& 1.000D-03& 2.000D-03& 1.565D-02& 3.245D-02& 1.008D-01& 2.024D-01\\
     4& 3.704D-07& 7.416D-07& 1.158D-03& 2.316D-03& 7.653D-04& 1.362D-03\\
     5& 7.123D-10& 1.457D-09& 4.543D-07& 9.086D-07& 1.691D-04& 3.381D-04\\
     6& 3.218D-11& 6.489D-11& 1.917D-13& 3.833D-13& 7.533D-09& 1.507D-08\\
     7& 5.275D-13& 1.082D-12& 2.045D-23& 4.079D-23& 4.429D-18& 8.859D-18\\
     8& 2.738D-14& 5.397D-14& 1.148D-25& 2.297D-25& 1.295D-32& 3.094D-31\\
     9& 7.942D-16& 1.619D-15& 6.963D-29& 1.437D-28& 2.893D-31& 4.553D-31\\
    10& 3.044D-17& 6.124D-17& 9.111D-31& 1.930D-29& 1.049D-30& 5.133D-32\\
\hline
\end{array}$$
\caption{\label{table5} Relative errors in the approximations $\widehat{T}^{(1)}_{2,n}[{\cal F}]$
and $\widehat{T}^{(2)}_{2,n}[{\cal F}]$  for $t=0.3$ in Example \ref{ex5}.
Here $E^{(s)}_{2,n}=|\widehat{T}^{(s)}_{2,n}[{\cal F}] -I_2(0.3)|/|I_2(0.3)|$.}
\end{table}

}
\end{example}

\begin{example}\label{ex6}{\em With $m=3$:
$$ I_3(t)=\intBar^1_0\frac{1+x-x^2}{(x-t)^3}\,dx=
\frac{(1+t-t^2)(1-2t)}{2[t(1-t)]^2}-\frac{1-2t}{t(1-t)}-\log\frac{1-t}{t}.$$
We have computed $I_3(t)$  for $t=0.3$.\\
The exact value of the integral is $I_3(0.3)=2.73546857952209343844408750481721$.
The results of the computation are given in Table \ref{table6}.

\begin{table}[ht]
$$ \begin{array}{||r|| c|c||c|c||c|c||}
\hline
k&E^{(1)}_{3,2^k},\,p=5&E^{(2)}_{3,2^k},\,p=5&E^{(1)}_{3,2^k},\,p=10&E^{(2)}_{3,2^k},\,p=10&
E^{(1)}_{3,2^k},\,p=15&E^{(2)}_{3,2^k},\,p=15\\
\hline\hline
 1& 1.034D-01& 1.359D-01& 3.659D-01& 8.073D-01& 5.726D-01& 9.778D-01\\
 2& 7.095D-02& 1.408D-01& 7.539D-02& 8.150D-02& 1.674D-01& 4.613D-01\\
 3& 1.115D-03& 2.229D-03& 6.928D-02& 1.372D-01& 1.266D-01& 2.306D-01\\
 4& 4.647D-10& 6.997D-09& 1.363D-03& 2.727D-03& 2.254D-02& 4.505D-02\\
 5& 7.926D-09& 1.618D-08& 5.015D-07& 1.003D-06& 3.534D-05& 7.069D-05\\
 6& 3.289D-10& 6.634D-10& 5.371D-14& 1.074D-13& 1.738D-09& 3.476D-09\\
 7& 5.653D-12& 1.159D-11& 7.388D-23& 1.470D-22& 1.898D-18& 3.795D-18\\
 8& 2.848D-13& 5.614D-13& 7.120D-25& 1.424D-24& 1.820D-28& 1.134D-27\\
 9& 8.178D-15& 1.667D-14& 4.612D-28& 1.352D-27& 1.490D-27& 9.074D-27\\
10& 3.150D-16& 6.336D-16& 6.468D-27& 1.888D-26& 1.211D-26& 6.680D-26\\
\hline
\end{array}$$
\caption{\label{table6} Relative errors in the approximations $\widehat{T}^{(1)}_{3,n}[{\cal F}]$
and $\widehat{T}^{(3)}_{2,n}[{\cal F}]$  for $t=0.3$ in Example \ref{ex6}.
Here $E^{(s)}_{3,n}=|\widehat{T}^{(s)}_{3,n}[{\cal F}] -I_3(0.3)|/|I_3(0.3)|$.}
\end{table}
}
\end{example}

\noindent{\bf Remarks.}
\begin{enumerate}
\item
Judging from Tables \ref{table1}--\ref{table6}, we may conclude that,  for each $m$, the quadrature formulas  $\widehat{T}{}^{(0)}_{m,n}[f]$
and $\widehat{T}{}^{(1)}_{m,n}[f]$ produce approximately the same accuracies. This is consistent with Theorem \ref{thw2} that says that $\widehat{T}{}^{(s)}_{m,n}[f]-I[f]=O(n^{-q})$ as $n\to\infty$  simultaneously for $s=0,1,\ldots$; that is, both formulas converge at the same rate as $n\to\infty$.
\item
The floating-point computation of HFP integrals is accompanied by roundoff errors that increase with $n$.
As discussed in Sidi \cite{Sidi:2014:RES} and \cite{Sidi:2019:SSI-P1},
these errors grow like (i)\,${\bf u}\log n$ when $m=1$, (ii)\,${\bf u}n$ when $m=2$, and (iii)\,${\bf u}n^2$ when $m=3$, where ${\bf u}$ is the roundoff unit of the floating-point arithmetic being used. The numbers at the bottom of Tables \ref{table1}--\ref{table6} (especially those corresponding to $n=2^8, 2^9,2^{10}$ with $p=10$ and $p=15$) exhibit this behavior since  ${\bf u}=1.93\times 10^{-34}$  for quadruple-precision arithmetic. Because the methods we have developed here converge quickly due to the fact that $q$ can be made as large as we wish, sufficiently high accuracy is achieved before roundoff errors start to be felt.   This is one important feature of our methods.
 \end{enumerate}

\appendix
\section*{Appendix: Further generalization of the Euler--Maclaurin \\ expansion}\label{AAA}
\setcounter{section}{1}
\setcounter{equation}{0} \setcounter{theorem}{0}

We begin with the  classical E--M expansion with remainder:
\begin{theorem} \label{thEM}
Let $w\in C^{p-1}[a,b]$, $w^{(p)}(x)$ being absolutely integrable on $[a,b]$, and let $h=(b-a)/n$, $n=1,2, \ldots.$ Then
\be\label{eqz1} h\sum^n_{j=0}{}^{''}w(a+jh)=\int^b_aw(x)\,dx+
\sum^p_{k=2}\frac{B_k}{k!}[w^{(k-1)}(b)-w^{(k-1)}(a)]h^k+R_p(h),\ee
where the remainder term $R_p(h)$ is given as
\be\label{eqz2} R_p(h)=-h^p\int^b_a w^{(p)}(x)\frac{\bar{B}_p(n\frac{a-x}{b-a})}{p!}\,dx.\ee
Here,  $B_k$ are  Bernoulli numbers, $\bar{B}_k(z)$  are periodic Bernoulli functions,\footnote{$\bar{B}_k(z)$ is the 1-periodic extension of the Bernoulli polynomial $B_k(z)$.} and
$$\displaystyle \sum^n _{j=0}{}^{''}\epsilon_j=\frac{1}{2}\epsilon_0+\sum^{n-1}_{j=1}\epsilon_j+
\frac{1}{2}\epsilon_n.$$
\end{theorem}

For more on the classical E--M expansion
with remainder, see Steffensen \cite{Steffensen:1950:I},  Ralston and Rabinowitz \cite[pp. 136--138]{Ralston:1978:FCN},
Stoer and Bulirsch \cite[pp. 156--159]{Stoer:2002:INA}, and  Sidi \cite[Appendix D]{Sidi:2003:PEM}, for example. In this work, we make use of the following corollary of Theorem \ref{thEM}:

\begin{corollary} \label{cor}
When $w^{(i)}(a)=w^{(i)}(b)=0$, $i=0,1,\ldots,p-1$, \eqref{eqz1}--\eqref{eqz2} reduce to
\be\label{eqz3}h\sum^{n-1}_{j=1}w(a+jh)=\int^b_aw(x)\,dx+R_p(h),\ee
with $R_p(h)=O(h^p)\quad\text{as $h\to 0$.}$
Actually,
\be\label{eqz35} \big|R_p(h)\big|\leq C_p\,h^p, \quad
C_p=\frac{1}{p!}\,\big(\max_{0\leq z\leq 1}\big|B_p(z)\big|\big)
\bigg(\int^b_a\big|w^{(p)}(x)\big|\,dx\bigg)<\infty.
\ee
\end{corollary}
\begin{proof}
First, because
$$ h\sum^n_{j=0}{}^{''}w(a+jh)= h\sum^{n-1}_{j=1}w(a+jh) \quad\text{since $w(a)=w(b)=0$},$$
and because
$$ \sum^p_{k=2}\frac{B_k}{k!}[w^{(k-1)}(b)-w^{(k-1)}(a)]h^k=0,\quad
\text{since $w^{(i)}(a)=w^{(i)}(b)=0$, $i=1,\ldots,p-1$,}$$
\eqref{eqz1} reduces to \eqref{eqz3}.
Next, \eqref{eqz35} follows by taking absolute values in \eqref{eqz2}.
\end{proof}

In \cite{Navot:1961:EEM}, Navot   extended Theorem \ref{thEM} to integrands $f(x)$ with an algebraic end-point singularity of the form $f(x)=(x-a)^\alpha g(x)$ with $\alpha>-1$, $g\in C^\infty[a,b]$. Using a different approach,  Lyness and Ninham \cite{Lyness:1967:NQA}, extended the   E--M expansion further to  singular integrands  of the form $f(x)=(x-a)^\alpha g_a(x)=(b-x)^\beta g_b(x)$ with $\alpha,\beta>-1$, $g_a\in C^\infty[a,b)$, $g_b\in C^\infty(a,b]$.
 Theorem 2.3 in Sidi \cite{Sidi:2012:EME-P1}, generalizes all the above in that (i)\, it applies to finite-range integrals  of integrands  that have {\em arbitrary} algebraic endpoint singularities and (ii)\,these  integrals can be  defined  in the regular sense or in the sense of HFP. Thus, it contains as special cases, but is not contained in, the classical E--M expansion and its extensions given in \cite{Navot:1961:EEM} and \cite{Lyness:1967:NQA}.

Theorem \ref{thA} below, which we use in this work, is an  extension of  Theorem 2.3 in  \cite{Sidi:2012:EME-P1}. Thus, it is also a further extension of the classical E--M expansion.

\begin{theorem}\label{thA}
Let $u\in C^{\infty}(I_1)$  and
$u\in C^{p-1}(I_2)$, $u^{(p)}(x)$ being absolutely integrable in $I_2$,
where
$$I_1=(a,a'')\cup(b'',b),\quad I_2=[a'',b''],\quad a<a''<b''<b,$$  and assume that $u(x)$ has the
asymptotic expansions
\be\label{eq37}\begin{split}
&u(x)\sim K(x-a)^{-1}+\sum^{\infty}_{s=0}c_s\,(x-a)^{\gamma_s}
\quad \text{as}\   x\to a+,\\
&u(x)\sim L(b-x)^{-1}+\sum^{\infty}_{s=0}d_s\,(b-x)^{\delta_s}
\quad \text{as}\  x\to b-,
\end{split}\ee
where the $\gamma_s$ and
$\delta_s$  are distinct complex numbers that satisfy
\begin{equation}\label{eq38}
\begin{matrix}
&\gamma_s\neq -1\quad \forall  s;\quad \text{\em Re\,}\gamma_0\leq\text{\em Re\,}\gamma_1\leq\text{\em Re\,}\gamma_2\leq\cdots;
&\lim_{s\to\infty}\text{\em Re\,}\gamma_s=+\infty,\\ \\
&\delta_s\neq -1\quad \forall s;\quad \text{\em Re\,}\delta_0\leq\text{\em Re\,}\delta_1\leq\text{\em Re\,}\delta_2\leq\cdots;
&\lim_{s\to\infty}\text{\em Re\,}\delta_s=+\infty.
\end{matrix}
\end{equation}
Assume furthermore that, for each positive integer $k$,
$u^{(k)}(x)$ has  asymptotic expansions as $x\to a+$ and
$x\to b-$ that are obtained by differentiating those of
$u(x)$  term by term $k$ times.\footnote{We express this briefly by saying that
``the asymptotic expansions in \eqref{eq37} can be differentiated infinitely  many times.''}
Let also $h=(b-a)/n$ for  $n=1,2,\ldots.$ Then, as $h\to0$,
\begin{align}
h\sum^{n-1}_{j=1}u(a+jh)\sim \intBar^b_au(x)\,dx+R_p(h)&+ K(C-\log h)+
\sum^{\infty}_{\substack{s=0\\ \gamma_s\not\in\{2,4,6,\ldots\}}}
c_s\,\zeta(-\gamma_s)\,h^{\gamma_s+1} \notag\\
&+L(C-\log h)+\sum^{\infty}_{\substack{s=0\\ \delta_s\not\in\{2,4,6,\ldots\}}}
d_s\,\zeta(-\delta_s)\,h^{\delta_s+1}, \label{eq39}
\end{align}
where $R_p(h)=O(h^p)$ as $h\to0$ and  $C=0.577\cdots$ is Euler's constant.\footnote{Note that the  constants $K$ and/or $L$ in \eqref{eq37} hence in  \eqref{eq39} can be zero.}
\end{theorem}

\noindent{\bf Remarks.}
\begin{enumerate}
\item
Note that if $K=L=0$ and $\Re\gamma_0>-1$ and $\Re\delta_0>-1$, then $\int^b_a u(x)\,dx$ exists as a regular integral. Otherwise, it does not, but its HFP $\intBar^b_a u(x)\,dx$ exists.
\item
When $u(x)$ is infinitely differentiable at $x=a$ and $x=b$, its Taylor series at $x=a$ and at $x=b$, whether convergent or divergent, are also (i)\,its asymptotic expansions  as $x\to a+$ and as $x\to b-$, respectively, and (ii)\,they can be differentiated term by term any number of times.
Thus, Theorem \ref{thA} applies without further assumptions on $u(x)$ in this case.
\item \label{re}
When $u\in C^\infty(a,b)$, we have that $p=\infty$; therefore, $R_p(h)$  is absent from \eqref{eq39} since its contribution is smaller than each of  the terms in the infinite sums there. That is, when $u\in C^\infty(a,b)$, the  generalization of the  E--M expansion is completely determined by the asymptotic expansions of $u(x)$ as $x\to a+$ and as $x\to b-$, nothing  else being needed. What happens in $(a,b)$ is immaterial. Precisely this result was obtained  in Sidi
\cite[Theorem 2.3]{Sidi:2012:EME-P1} and we shall make use of it when proving Theorem \ref{thA}. Thus,  Theorem \ref{thA} is a nontrivial  extension of  Theorem 2.3 in \cite{Sidi:2012:EME-P1}.
\item
It is clear from \eqref{eq39}  that the positive even powers of
$(x-a)$ and  $(b-x)$,
if present in the asymptotic expansions of $u(x)$ as
$x\to a+$ and $x\to b-$,
do not contribute to the asymptotic expansion of
$h\sum^{n-1}_{j=1}u(a+jh)$ as $h\to 0$, the reason being that
$\zeta(-2k)=0$ for $k=1,2,\ldots.$  We have included the ``limitations''
$\gamma_s\not\in\{2,4,6,\ldots\}$ and $\delta_s\not\in\{2,4,6,\ldots\}$ in the sums on the right-hand side of \eqref{eq39} only as ``reminders.''

\item Theorem 2.3 in \cite{Sidi:2012:EME-P1} is only a special case of a more general theorem there
 involving the so-called {\em offset trapezoidal rule}
$h\sum^{n-1}_{i=0}u(a+jh+\theta h)$, with $\theta\in(0,1]$ fixed,\footnote{Note that, with $\theta=1/2$, the offset trapezoidal rule becomes the mid-point rule.}
    that contains as special cases  all previously known generalizations of the E--M expansions
    for integrals with {\em algebraic} endpoint singularities. For  a further generalization
pertaining to arbitrary {\em algebraic-logarithmic} endpoint singularities, see
    Sidi \cite{Sidi:2012:EME-P2}.
\end{enumerate}

\begin{proof} To prove \eqref{eq39}, we begin by constructing two so-called {\em neutralizers} $R_\pm(x)\in C^\infty[a,b]$, such that $R_+(x)+R_-(x)=1$ for all $x$, as follows: Choosing $a',b'$ such that
$$ a<a'<a''<b''<b'<b,$$ we let
\begin{align*}
&R_+(x)=1\quad\text{for $x\in[a,a']\cup[b',b]$};\quad
&R_+(x)=0\quad\text{for $x\in[a'',b'']$}, \\
&R_-(x)=0\quad\text{for $x\in[a,a']\cup[b',b]$};\quad
&R_-(x)=1\quad\text{for $x\in[a'',b'']$},
\end{align*}
such that $R_+(x)$ ($R_-(x)$) decreases (increases) on $(a',a'')$ and increases (decreases)
on $(b'',b')$ and
 \beq\label{Rder}
   R_\pm^{(i)}(a')=R_\pm^{(i)}(a'')=R_\pm^{(i)}(b'')=R_\pm^{(i)}(b')=0,\quad i=1,2,\ldots.\eeq 

 With the functions $R_\pm(x)$ available, we now split $u(x)$ as  in
 \beq\label{sumf}u(x)=u_+(x)+u_-(x);\quad
 u_+(x)=R_+(x)u(x),\quad u_-(x)=R_-(x)u(x).\eeq

 First,  $u_+(x)\equiv u(x)$ for $x\in[a,a']\cup[b',b]$ and
 $u_+(x)\equiv0$ for $x\in[a'', b'']$; therefore, $u_+\in C^\infty(a,b)$
 and has the asymptotic expansions given in \eqref{eq37}. Consequently,
 Theorem 2.3 in \cite{Sidi:2012:EME-P1} applies (recall Remark \ref{re} following the statement of Theorem \ref{thA}), and we have, as $h\to0$,

 \begin{align}
h\sum^{n-1}_{j=1}u_+(a+jh)\sim\intBar^b_au_+(x)\,dx&+ K(C-\log h)+
\sum^{\infty}_{\substack{s=0\\ \gamma_s\not\in\{2,4,6,\ldots\}}}
c_s\,\zeta(-\gamma_s)\,h^{\gamma_s+1} \notag\\
&+L(C-\log h)+\sum^{\infty}_{\substack{s=0\\ \delta_s\not\in\{2,4,6,\ldots\}}}
d_s\,\zeta(-\delta_s)\,h^{\delta_s+1}\label{eq39a}
\end{align}

Next, $u_-(x)\equiv u(x)$ for $x\in[a'',b'']$ and $u_-(x)\equiv0$ for
$x\in[a,a']\cup[b',b]$; therefore, $u_-\in C^{p-1}[a,b]$
and $u_-^{(i)}(a)=u_-^{(i)}(b)=0$, $i=0,1,\ldots,p-1,$
and $u_-^{(p)}(x)$ is absolutely integrable in $[a,b]$.
Consequently,  Corollary \ref{cor}  applies to  $\intBar^b_au_-(x)\,dx$, which is now the regular integral
 $\int^b_au_-(x)\,dx$, and we have
\be \label{eq47} h\sum^{n-1}_{j=1}u_-(a+jh)=\int^b_au_-(x)\,dx+R_p(h),\quad
R_p(h)=O(h^p)\quad \text{as $h\to0$}.\ee

Finally,  adding \eqref{eq47} to \eqref{eq39a}, noting that
$$h\sum^{n-1}_{j=1}u_+(a+jh)+h\sum^{n-1}_{j=1}u_-(a+jh)=h\sum^{n-1}_{j=1}u(a+jh),$$
and  recalling also  that
$$\intBar^b_au_+(x)\,dx+\intBar^b_au_-(x)\,dx=\intBar^b_au(x)\,dx,$$
we obtain \eqref{eq39}. This completes the proof.
\end{proof}

\section*{Acknowledgement} The author would like to thank
Mr. Eitan Kaminski for producing the  graphs included in this paper.


\newpage
\vspace{-12cm}
\begin{center}
\begin{figure}[!]
\includegraphics[width=\columnwidth]{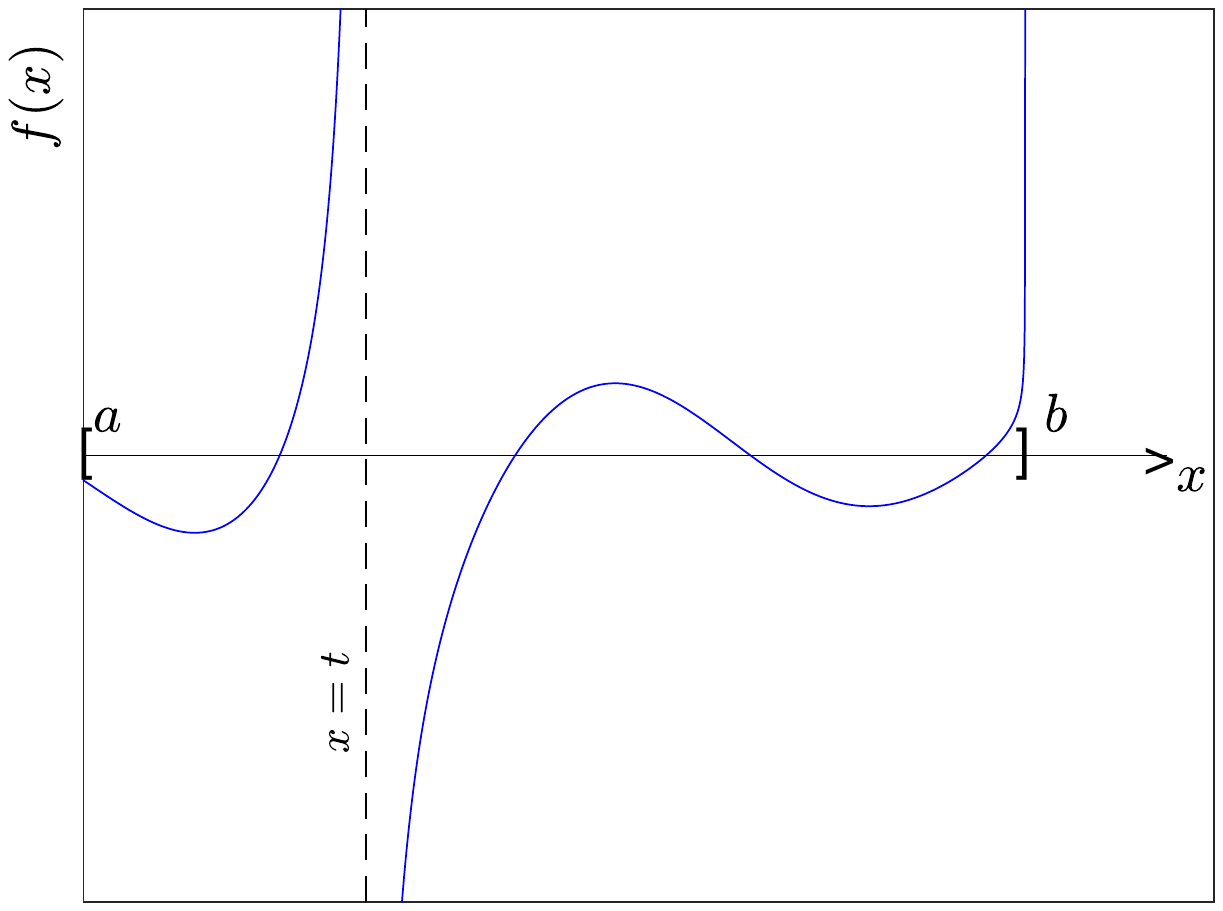}
\vspace{-7cm}
\caption{\label{fig1} Graph of the function $f(x)=\sin(4\pi x+\pi/6)(1-x)^{-1/3}/(x-t)$ on
$[a,b]=[0,1]$, with $t=0.3$.}
\end{figure}
\end{center}

\begin{center}
\begin{figure}[!]
\includegraphics[width=1.2\textwidth]{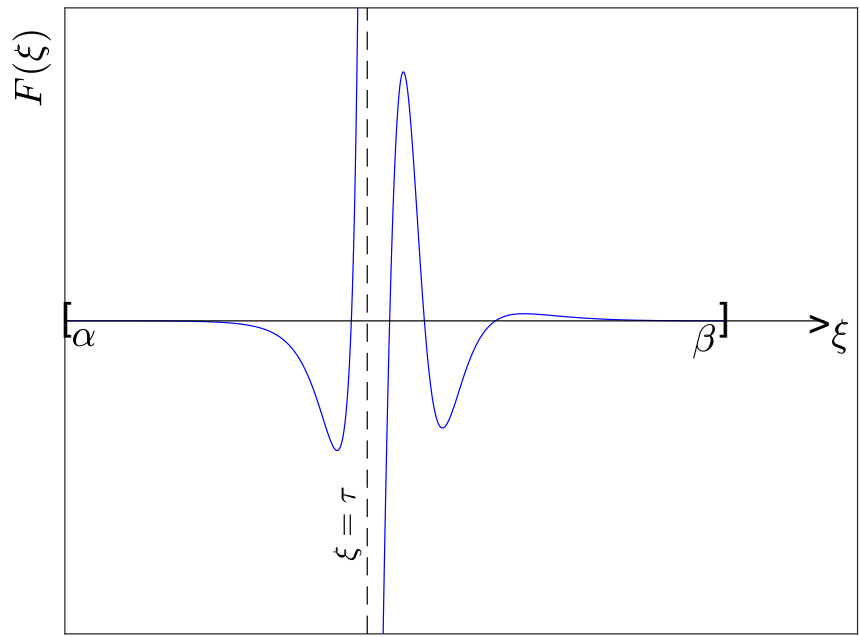}
\vspace{-7cm}
\caption{\label{fig2} Graph of the function ${F}(\xi)=f(\hpsi(\xi))\hpsi'(\xi)$
on $[\alpha,\beta]=[0,1]$. Here $f(x)$ is the function in Figure \ref{fig1},  $\hpsi(\xi)$ is as in \eqref{eqc3}  with $p=5$,  and $\tau$ is given by \eqref{eqtau1} with $p=5$ and $t=0.3$ there.}
\end{figure}
\end{center}

\begin{center}
\begin{figure}[!]
\includegraphics[width=1.2\textwidth]{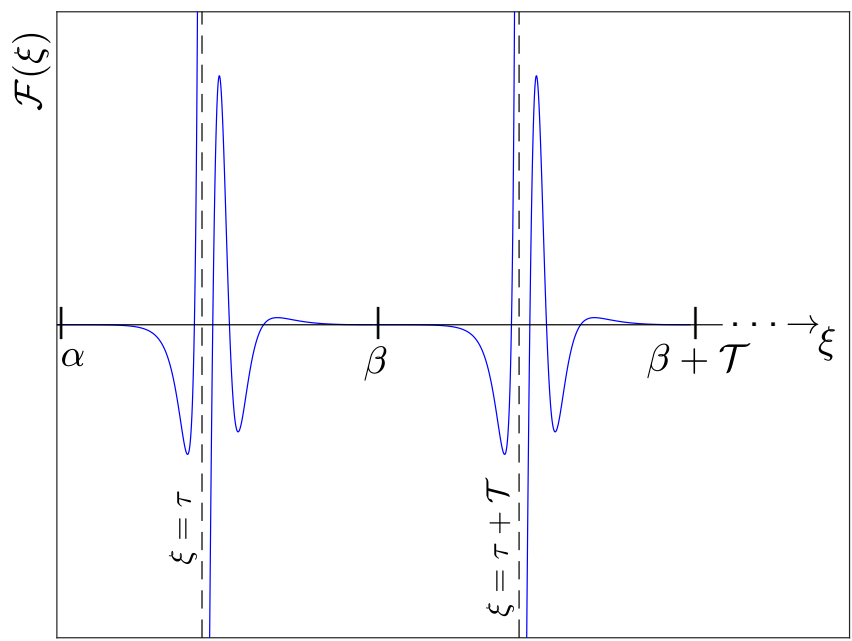}
\vspace{-7cm}
\caption{\label{fig3} Graph of the function $\mathcal{F}(\xi)$,
the $\mathcal{T}$-periodic extension of the function $F(\xi)=f(\hpsi(\xi))\hpsi'(\xi)$ in Figure \ref{fig2}, with $\mathcal{T}=\beta-\alpha=1$.}
\end{figure}
\end{center}

\end{document}